\documentclass{article}
\usepackage{amsmath,amsthm,amssymb,float,graphicx,geometry}
\usepackage{bbm,bbding,amssymb,pifont,mathrsfs,amsfonts,graphicx,subfigure,mathtools,color}%
\usepackage{amscd,array,enumerate,dsfont,texdraw,tikz,multicol,bbm,authblk}
\usepackage{algorithmic}
\newtheorem{algorithm}{Algorithm}
\usepackage{algorithm}
\usepackage{footnote}
\usepackage{lipsum}
\usepackage{amsmath}
\usepackage{cases}
\usepackage{fancyhdr}
\usepackage{CJK}
\usepackage{bm}
\usepackage{framed}
\usepackage{listings}
\usepackage{multirow}
\allowdisplaybreaks[4]
\usetikzlibrary{shapes.geometric, arrows}

\newtheorem{assumption}{Assumption}
\newtheorem{lem}{Lemma}
\newtheorem{thm}{Theorem}
\newtheorem{Def}{Definition}
\newtheorem{Col}{Corollary}
\newtheorem{remark}{Remark}
\newtheorem{example}{Example}

\tikzstyle{startstop} = [rectangle, rounded corners, minimum width = 2cm, minimum height=1cm,text centered, draw = black, fill = red!40]
\tikzstyle{io} = [rectangle, rounded corners, minimum width = 2cm, minimum height=1cm,text centered, draw = black, fill = blue!40]
\tikzstyle{process} = [rectangle, rounded corners, minimum width = 2cm, minimum height=1cm,text centered, draw = black, fill = yellow!50]
\tikzstyle{decision} = [rectangle, rounded corners, minimum width = 2cm, minimum height=1cm,text centered, draw = black, fill = green!40]
\tikzstyle{arrow} = [->,>=stealth]
\title{No-Regret Learning in Network Stochastic Zero-Sum Games}
\author[a,b]{Shijie Huang}
\author[c]{Jinlong Lei}
\author[c,a]{Yiguang Hong}
\affil[a]{Key Laboratory of Systems and Control, Academy of Mathematics and Systems Science, Chinese Academy of Sciences}
\affil[b]{School of Mathematical Sciences, University of Chinese Academy of Sciences}
\affil[c]{Department of Control Science and Engineering \& Shanghai Research Institute for Intelligent Autonomous Systems, Tongji University}

\date{}
\begin{document}
\maketitle
\definecolor{shadecolor}{rgb}{0.9,0.9,0.9}
\begin{abstract}
	No-regret learning has been widely used to compute a Nash equilibrium in two-person zero-sum games. However, there is still a lack of regret analysis for network stochastic zero-sum games, where players competing in two subnetworks only have access to some local information, and the cost functions include uncertainty. Such a game model can be found in security games, when a group of inspectors work together to detect a group of evaders. In this paper, we propose a distributed stochastic mirror descent (D-SMD) method, and establish the regret bounds $O(\sqrt{T})$ and $O(\log T)$ in the expected sense for convex-concave and strongly convex-strongly concave costs, respectively. Our bounds match those of the best known first-order online optimization algorithms. We then prove the convergence of the time-averaged iterates of D-SMD to the set of Nash equilibria. Finally, we show that the actual iterates of D-SMD almost surely converge to the Nash equilibrium in the strictly convex-strictly concave setting.
\end{abstract}



\section{INTRODUCTION}
Two-person zero-sum games \cite{mv:53} are ubiquitous and well-researched topics in economics, convex optimization and robust optimization \cite{ben:09}. They are related to a variety of artificial intelligence problems, such as boosting \cite{fs:96}, generative adversarial networks (GAN) \cite{gp:14}, and poker games \cite{bb:15,ms:17}. So far, researchers have mainly focused on computing a Nash equilibrium (NE) \cite{n:51} and made significant progress. Specifically, no-regret learning has proved to be an extremely versatile tool in this direction. For instance, some typical no-regret algorithms such as follow the regularized leader, mirror descent (MD) and its variants \cite{rs:13a,ahk:12,z:17}, have become popular for finding a NE of a two-person zero-sum game. More recently, these algorithms have paved the way to designing algorithms with faster rates for both regret and convergence to a NE \cite{rs:13b,ddk:11,kh:18}.

However, those algorithms rely on having access to complete information of the players. In practice, we may encounter the class of network games such as network zero-sum games, where the players are partitioned into two subnetworks and each player only has access to local information \cite{gharesifard2013distributed,lh:15}. The security game involving a group of evaders and a group of inspectors \cite{cc:16} could be an example. Additionally, many game problems in machine learning, such as GAN and model-based reinforcement learning \cite{rmk:20}, are complicated by uncertainty. Such problems may be modeled by stochastic Nash games, where the cost functions are expectation-valued. In this paper we consider no-regret learning in network stochastic zero-sum games.

Compared with two-person zero-sum games, no-regret learning in network stochastic zero-sum games is more challenging. First of all, one needs to define a regret different from the classical regret due to the absence of complete information, which also brings difficulties to the regret analysis. Although the time-averaged iterates of no-regret learning algorithms are guaranteed to converge to a NE in two-person zero-sum games \cite{rs:13b}, whether such a property can be extended to a network stochastic zero-sum game remains unexplored. In addition, each player cannot accurately evaluate its (sub)gradient since the cost functions are expectation-valued. As a consequence, convergence analysis in a network stochastic zero-sum game may require different techniques.
\subsection{Contributions}
In this work, we propose a distributed stochastic mirror descent (D-SMD) method for network stochastic zero-sum games, and establish the theoretical results regarding the regret bounds and convergence guarantees. We elaborate on our contributions below.
\begin{itemize}
	\item {\it Regret bounds of D-SMD}: We show that D-SMD achieves the regret bounds of $O(\sqrt{T})$ and $O(\log T)$ in the convex-concave and strongly convex-strongly concave cases, respectively. Despite the influence of the network parameters, our results match the regret order of MD in the convex and strongly convex cases \cite{s:11,ss:07}.
	\item {\it Convergence guarantees of D-SMD}: We establish the mean convergence of the time-averaged iterates to the set of Nash equilibria in the convex-concave and strongly convex-strongly concave settings, and provide the convergence rates. In addition, we also prove that the iterates of D-SMD converge to the unique NE with probability one in the strictly convex-strictly concave case.  
\end{itemize}
For the sake of readability, we have moved all the omitted proofs to the appendix.

\subsection{Related Work}
We briefly review two kinds of related works: no-regret learning in games and NE seeking in zero-sum games.

{\bf No-regret learning in games.} No-regret algorithms for two-person games were first proposed by \cite{h:57} and \cite{b:56}, and further studied in the work of \cite{fv:99}. In \cite{fs:99}, the regret bound of the multiplicative weights algorithm was established, which immediately yields $O(T^{-\frac{1}{2}})$ convergence rate. To obtain a faster algorithm, Nesterov's excessive gap technique was adopted to exhibit a near-optimal algorithm with convergence rate $O(\frac{(\ln T)^{3/2}}{T})$ in \cite{ddk:11}. Later a simpler no-regret framework with rate $O(\frac{\ln T}{T})$ based on the optimistic MD was proposed by \cite{rs:13b} and this rate was further improved to $O(\frac{1}{T})$ by \cite{kh:18}. In addition, no-regret learning was also widely studied in zero-sum extensive-form games, due to the success of CFR framework \cite{zj07} and its variants \cite{lw:09,t:14} in solving the game of limit Texas hold'em \cite{bb:15}. Recently, no-regret learning was extended to other form of games. \cite{hst:15} considered no-regret learning and its outcomes in Bayesian games, while \cite{sbk:19} proposed a no-regret algorithm for unknown games with correlated payoffs. \cite{mz:19} studied the last-iterate convergence of a no-regret algorithm to a NE of variationally stable games and \cite{l:20} further proved the finite-time last-iterate convergence rate of online gradient descent learning in cococercive games.   

{\bf NE seeking in zero-sum games.} There is a large amount of work on computing a NE of zero-sum games (or saddle point problems). We recall some works that focus on games with continuous strategy sets. \cite{hs:06} established the convergence of continuous-time best response dynamics to the NE set. \cite{nj:09} proposed a stochastic MD scheme to solve a convex-concave saddle point problem. Later on, \cite{clo:14} presented a stochastic accelerated primal-dual method with optimal convergence rate. In \cite{pb:16} and \cite{ykh:20}, the authors studied the computation of saddle points in strongly convex-strongly concave and non-convex-non-concave settings, respectively. Moreover, \cite{lei2020synchronous} considered the Nash equilibrium computation for general-sum stochastic Nash games. \cite{b:21} designed a primal-dual algorithm for constrained markov decision process (equivalent to a saddle point problem) and utilized regret analysis to prove zero constraint violation. In a different line of research, \cite{srivastava2013distributed} considered the case when a network of cooperative agents need to solve a zero-sum game and proposed a distributed Bregman-divergence algorithm to compute a NE. \cite{gharesifard2013distributed} introduced a continuous-time distributed dynamics for a more general framework of network zero-sum games, where two network of agents are involved in a zero-sum game. \cite{lh:15} further extended the framework to time-varying networks and designed a distributed projected subgradient descent algorithm.

\section{Preliminaries \& Problem Formulation}

{\bf Notations.} For a matrix $A = [a_{ij}]$, $a_{ij}$ denotes the element in the $i$th row and $j$th column. For a function $f(x_1,\dots,x_N)$, we use $\partial_i f$ to denote the subdifferential of $f$ with respect to $x_i$. Given a norm $\|\cdot\|$ on $\mathbb{R}^n$, $\|y\|_{\ast}:= \sup_{\|x\|\le 1}\langle y,x\rangle$ denotes the dual norm. A digraph is characterized by $\mathcal{G} = (\mathcal{V},\mathcal{E})$, where $\mathcal{V} = \{1,\dots,n\}$ is the set of nodes and $\mathcal{E}\subset\mathcal{V}\times\mathcal{V}$ is the set of edges. A path from $i_1$ to $i_p$ is an alternating sequence of edges $(i_1,i_2),(i_2,i_3),\dots, (i_{p-1},i_p)$ in the digraph with distinct nodes $i_m\in\mathcal{V}$, $\forall m: 1\le m\le p$. A digraph is strongly connected if there is a path between any pair of distinct nodes.

\subsection{Two-Network Stochastic Zero-Sum Game}

{\bf Two-Person Zero-Sum Game and Nash Equilibrium.} We recall the definition of a two-person zero-sum game and the sufficient conditions to ensure the existence of a Nash equilibrium.
\begin{Def}
	A two-person zero-sum game consists of two players who select strategies from nonempty sets $X_1$ and $X_2$, respectively. Players observe a cost function $U_1: X_1\times X_2\to\mathbb{R}$ and $U_2: X_1\times X_2\to\mathbb{R}$ that satisfy  $U_1(x_1,x_2) + U_2(x_1,x_2) = 0$ for all $(x_1,x_2)\in X_1\times X_2$. 
\end{Def} 
Define $U:= U_1 = -U_2$ as the cost function of the game. The most widely used solution concept in non-cooperative games is that of a Nash equilibrium (NE), which is formally defined as follows.
\begin{Def}
	A strategy profile $x^{\ast} = (x_1^{\ast},x_2^{\ast})$ is a Nash equilibrium (NE) of a two-person zero-sum game if $U(x_1^{\ast},x_2)\le U(x_1^{\ast},x_2^{\ast})\le U(x_1,x_2^{\ast})$ for all $(x_1,x_2)\in X_1\times X_2$.
\end{Def}
\begin{thm}[Existence of NE \cite{gharesifard2013distributed}]
	Suppose that the strategy sets $X_1$ and $X_2$ are compact and convex. If the cost function $U$ is continuous and convex-concave (convex in $x_1$, concave in $x_2$) over $X_1\times X_2$, then there exists a NE for the considered two-person zero-sum game.
\end{thm}

{\bf A Two-Network Zero-Sum Game} \cite{gharesifard2013distributed,lh:15} is a generalized two-person zero-sum game, defined by a tuple $(\{\Sigma_1,\Sigma_2\},X_1\times X_2,U)$. The two players $\Sigma_1$ and $\Sigma_2$ are directed networks composed of $n_1$ agents and $n_2$ agents. For $l = 1,2$, $X_l\subset\mathbb{R}^{m_l}$, denoting the strategy set of $\Sigma_l$, is assumed to be compact and convex. The cost function $U: X_1\times X_2\to\mathbb{R}$ is defined by
\[U(x_1,x_2) = \frac{1}{n_1}\sum_{i=1}^{n_1}f_{1,i}(x_1,x_2) = -\frac{1}{n_2}\sum_{j=1}^{n_2}f_{2,j}(x_1,x_2),\]
where $f_{1,i}$ is a convex-concave continuous cost function associated with agent $i$ in $\Sigma_1$ and $f_{2,j}$ is a concave-convex continuous cost function associated with agent $j$ in $\Sigma_2$. The networks have no global decision-making capability and each agent only knows its own cost function. Within the same network, neighboring agents can exchange information. Moreover, the interaction between the two networks is specified by a bipartite network $\Sigma_{12}$, which means that each network can also obtain information about the other network through $\Sigma_{12}$. The goal of the agents in $\Sigma_1$ ($\Sigma_2$) is to collaboratively minimize (maximize) the cost function $U$ based on local information.

More precisely, $\Sigma_1$, $\Sigma_2$ and $\Sigma_{12}$ are described by three directed graph sequences $\mathcal{G}_1(t) = (\mathcal{V}_1,\mathcal{E}_1(t))$, $\mathcal{G}_2(t) = (\mathcal{V}_2,\mathcal{E}_2(t))$ and $\mathcal{G}_{12}(t) = (\mathcal{V}_1\cup\mathcal{V}_2,\mathcal{E}_{12}(t))$, where $\{\mathcal{G}_1(t)\}$ and $\{\mathcal{G}_2(t)\}$ are uniformly jointly strongly connected\footnote{Namely, there exists an integer $B_l\ge 1$ such that the union graph $(\mathcal{V}_l,\bigcup_{t=k}^{k+B_l-1}\mathcal{E}_l(t))$ is strongly connected for $k\ge 0$.}. The communication between agents are modeled by mixing matrices $W_1(t)$, $W_2(t)$ and $W_{12}(t)$, which satisfy (i) for each $l\in\{1,2\}$, $w_{l,ij}(t)\ge\eta$ with $0 < \eta < 1$ when $(j,i)\in\mathcal{E}_l(t)$, and $w_{l,ij}(t) = 0$ otherwise; $w_{12,ij}(t) > 0$ only if $(i,j)\in\mathcal{E}_{12}(t)$;
(ii) for each $i,j\in\mathcal{V}_l$, $\sum_{j = 1}^{n_l}w_{l,ij}(t) = \sum_{i = 1}^{n_l}w_{l,ij}(t) = 1$; (iii) for each $i\in\mathcal{V}_l$, $\sum_{j = 1}^{n_{3-l}}w_{12,ij}(t) = 1$. Agent $i\in\mathcal{V}_l$ can only communicate directly with its neighbors $\mathcal{N}_l^i(t) := \{j\mid(j,i)\in\mathcal{E}_l(t)\}$ and $\mathcal{N}_{12,l}^i(t) \triangleq \{j\mid (j,i)\in\mathcal{E}_{12}(t)\}$. \cite{no:10} proved the following result.
\begin{lem}\label{lem_graph}
	Let $\Phi_l(t,s) = W_l(t)W_l(t-1)\cdots W_l(s)$ ($l = 1,2$) be the transition matrices. Then for all $t,s$ with $t\ge s\ge 0$, we have
	\[\bigg|[\Phi_l(t,s)]_{ij} - \frac{1}{n_l}\bigg| \le \Gamma_l\theta_l^{t-s},\quad l = 1,2\]
	where $\Gamma_l = (1 - \eta/4n_l^2)^{-2}$ and $\theta_l = (1 - \eta/4n_l^2)^{1/B_l}$.
\end{lem}

{\bf Two-Network Stochastic Zero-Sum Game.} Consider a stochastic generalization of a two-network zero-sum game, where the cost function of each agent is expectation-valued. To be specific, we assume that $f_{l,i}$ is the expected value of a stochastic mapping $\psi_{l,i}: X_1\times X_2\times\mathbb{R}^d\to\mathbb{R}$, i.e.,
\[f_{l,i}(x_1,x_2):=\mathbb{E}[\psi_{l,i}(x_1,x_2;\xi(\omega))],\quad l\in\{1,2\}\]
where the expectation is taken with respect to the random vector $\xi: \Omega\to\mathbb{R}^d$ defined on a probability space $(\Omega,\mathcal{F},\mathbb{P})$. Such stochastic models represent a natural extension of two-network zero-sum games and find their applicability when the evaluation of the deterministic cost function is corrupted by errors. However, deterministic methods cannot be used to solve a two-network stochastic zero-sum game directly since generally the expectation cannot be evaluated efficiently or the underlying distribution $\mathbb{P}$ is unknown. This characteristic also makes the analysis of algorithm performance more complicated.

\subsection{No-Regret Learning}
In a no-regret learning framework \cite{z:03}, for $l = 1,2$, each agent $i\in\mathcal{V}_l$ plays repeatedly against the agents in $\Sigma_{3-l}$ by making a sequence of decisions from $X_l$. At each round $t = 1,\dots, T$ of a learning process, each agent $i$ in $\Sigma_l$ selects a strategy $x_{l,i}(t)\in X_l$ based on the available information, and receives a cost $f_{1,i}(x_{1,i}(t),u_{2,i}(t))$, where $u_{2,i}(t) \triangleq \sum_{j\in\mathcal{N}_{12,1}^i(t)}w_{12,ij}(t)x_{2,j}(t)$ is the weighted information received from its neighbors $\mathcal{N}_{12,1}^i(t):=\{j\in\mathcal{V}_2|(j,i)\in\mathcal{E}_{12}(t)\}$. Since $x_{1,i}(t)$ and $u_{2,i}(t)$ are generated with noisy information, we consider a notation different from the classical regret, called pseudo regret \cite[Section 2.1.2]{m:19}.
\begin{Def}
	The pseudo regret of $\Sigma_1$ associated with agent $i$ cumulated up to time $T$ is defined as
	\begin{align}
		\quad\bar{R}_1^{(i)}(T)
		&= \mathbb{E}\left[\sum_{t=1}^TU(x_{1,i}(t),u_{2,i}(t))\right]\notag\\
		&\quad - \min_{x_1\in\mathcal{X}_1}\mathbb{E}\left[\sum_{t=1}^TU(x_1,u_{2,i}(t))\right].\label{regret_def}
	\end{align}
\end{Def}
Intuitively, $\bar{R}_1^{(i)}(T)$ represents the maximum expected gain agent $i\in\Sigma_1$ could have achieved by playing the single best fixed strategy in case the estimated sequence of $\Sigma_2$'s strategies $\{u_{2,i}(t)\}_{t=1}^T$ and the cost functions were known in hindsight. An algorithm is referred to as no-regret for network $\Sigma_1$ if for all $i$, $\bar{R}_1^{(i)}(T)/T\to 0$ as $T\to\infty$.

\section{Proposed D-SMD Algorithm}
Our algorithm uses the notion of {\it prox-mapping}. For $l = 1,2$, $x,p\in X_l$, let the Bregman divergence be defined by
\begin{equation}\label{breg_def}
	D_{\psi_l}(x,p) := \psi_l(x) - \psi_l(p) - \langle\nabla\psi_l(p),x - p\rangle,
\end{equation}
where $\psi_l$ is a $1$-strongly convex differentiable regularizer on $\mathcal{X}_l$. Then the Bregman divergence generates an associated prox-mapping defined as
\begin{equation}\label{prox_mapping}
	P_x^l(y) := \arg\min_{x'\in X_l}\{\langle y,x - x'\rangle + D_{\psi_l}(x',x)\}.
\end{equation}
Two typical examples of prox-mapping includes Euclidean projection and multiplicative weights \cite{s:11}.
\begin{example}\label{exm1}
	Let $\psi_l(x) = \frac{1}{2}\|x\|_2^2$. Then the associated prox-mappping is 
	\[P_x^l(y) = \arg\min_{x'\in X_l}\|x' - x - y\|^2.\] 
\end{example}
\begin{example}\label{exm2}
	Let $X_l$ be a $d_l$-dimension simplex and $\psi_l(x) = \sum_{j=1}^{d_l}x_j\log x_j$ be the entropic regularizer. The induced prox-mapping becomes the well-known multiplicative weights rule \cite{fs:99}
	\[P_x^l(y) = \frac{(x_j\exp(y_j))_{j=1}^{d_l}}{\sum_{j=1}^{d_l}x_j\exp(y_j)}.\]
\end{example}

The distributed stochastic mirror descent algorithm proceeds as follows. In the $t$-th iteration, each agent replaces its local estimates of the states of $\Sigma_1$ and $\Sigma_2$ with the weighted averages of its neighbors in $\Sigma_1$ and $\Sigma_2$, respectively. Then it calculates a sampled subgradient of its local cost at the replaced estimates, and updates its estimate by a prox-mapping. The complete algorithm is summarized in Algorithm \ref{alg1}.  

\begin{algorithm}[tb]
	\caption{Distributed Stochastic Mirror Descent (D-SMD)}
	\label{alg1}
	\textbf{Input}: Non-increasing nonnegative step-size sequence $\{\alpha(t)\ge 0\}_{t\ge 0}$, mixing matrices $\{W_1(t)\}_{t\ge 0}$, $\{W_2(t)\}_{t\ge 0}$ and $\{W_{12}(t)\}_{t\ge 0}$\\
	\textbf{Initialize}: $x_{l,i}(0)\in\mathcal{X}_l$ for each $i \in\mathcal{V}_l$ and $l=1,2$
	\begin{algorithmic}[1] 
		\REPEAT 
		\FOR {network $l=1,2$}
		\FOR {agent $i = 1,\dots,n_l$}
		\STATE Calculate weighted average of neighbors in $\Sigma_l$ and $\Sigma_{3-l}$:\\
		$v_{l,i}(t) = \sum_{j=1}^{n_l}w_{l,ij}(t)x_{l,j}(t)$\\
		$u_{3-l,i}(t) = \sum_{j=1}^{n_{3-l}}w_{12,ij}(t)x_{3-l,j}(t)$\\
		\STATE Receive a sampled subgradient $\hat{g}_{l,i}(t)\in\partial_l\psi_{l,i}(v_{l,i}(t),u_{3-l,i}(t);\xi_{l,i}(t))$, where the terms $\xi_{l,i}(t)$ are independent and identically distributed realizations of the random variable $\xi$.\\
		\STATE Update local estimates $x_{l,i}(t+1)$:\\
		$x_{l,i}(t+1) = P_{v_{l,i}(t)}^l(-\alpha(t)\hat{g}_{l,i}(t))$
		\ENDFOR
		\ENDFOR
		\UNTIL Convergence
	\end{algorithmic}
\end{algorithm}

\begin{remark}
	To compute a NE of network stochastic zero-sum games, the stochastic mirror descent method for convex-concave saddle point problems \cite{nj:09} requires the global information of the network. While in our algorithm, each agent merely communicates
	its decisions with its neighbors to update the estimates .   
\end{remark}

Let $\mathcal{F}_t := \sigma\{x_{l,i}(0),\xi_{l,i}(s), l = 1,2, i\in\mathcal{V}_l, 0\le s\le t-1\}$ denote the $\sigma$-algebra generated by all the information up to time $t-1$. Then, by Algorithm \ref{alg1}, $x_{l,i}(t)$, $v_{l,i}(t)$ and $u_{l,i}(t)$ are adapted to $\mathcal{F}_t$. Denote by $g_{1,i}(t)\in\partial_1f_{1,i}(v_{1,i}(t),u_{2,i}(t))$ and $g_{2,i}(t)\in\partial_2f_{2,i}(u_{1,i}(t),v_{2,i}(t))$ the subgradients of $f_{1,i}$ and $f_{2,i}$ evaluated at $v_{1,i}(t)$ and $v_{2,i}(t)$. In the following, We draw three necessary assumptions, which are all standard and widely used in stochastic approximation and distributed optimization \cite{nj:09,srivastava2013distributed}.
\begin{assumption}\label{asm1}
	For $l = 1,2$, $i\in\mathcal{V}_l$, the cost function $f_{l,i}(\cdot,\cdot)$ is Lipschitz continuous over $X_1\times X_2$, i.e., there exists a constant $L > 0$ such that, for all $x_1,x_1'\in X_1$ and $x_2,x_2'\in X_2$,
	$$|f_{l,i}(x_1,x_2) - f_{l,i}(x_1',x_2')|\le L(\|x_1 - x_1'\| + \|x_2 - x_2'\|).$$
\end{assumption}
\begin{assumption}\label{asm2}
	There exists $\nu_l > 0$ such that for $l = 1,2$, and each $i\in\mathcal{V}_l$,
	\[\mathbb{E}[\hat{g}_{l,i}(t)|\mathcal{F}_t] = g_{l,i}(t)\  \text{and}\ \ \mathbb{E}[\|g_{l,i}(t) - \hat{g}_{l,i}(t)\|_{\ast}^2|\mathcal{F}_t]\le \nu_l^2.\]
\end{assumption}
\begin{assumption}\label{asm3}
	For $l = 1,2$, the Bregman divergence $D_{\psi_l}(x,y)$ is convex in $y$ and satisfies
	\[x_k\to x\quad\Rightarrow \quad D_{\psi_l}(x_k,x)\to 0\footnote{The regularizers mentioned in Examples \ref{exm1} and \ref{asm2} both satisfy this condition, which is called Bregman reciprocity \cite{mz:19}.}.\]
\end{assumption}

\section{Guarantees on Regret}
In this section, we establish regret bounds of D-SMD in convex-concave and strongly convex-strongly concave settings. 
\subsection{Convex-Concave Case}
This part presents a regret bound that holds at all time $T$ for D-SMD when the cost function $f_{1,i}(\cdot,\cdot)$ is convex-concave for all $i\in\mathcal{V}_1$.

Theorem \ref{thm1} provides a general bound for any choice of (non-increasing) step-size sequence $\{\alpha(t)\}_{t=1}^T$. It will then be the basis for Corollary \ref{col1}, which gives a way to select the algorithm parameters to achieve a sublinear regret.  
\begin{thm}\label{thm1}
	Let the cost function $f_{1,i}(\cdot,\cdot)$ be convex-concave and Assumptions \ref{asm1}-\ref{asm3} hold. Then the pseudo regret of D-SMD defined by \eqref{regret_def} is bounded by
	\begin{align}
		\bar{R}_1^{(i)}(T)&\le \sum_{t=1}^T\sum_{l=1}^2(L+\nu_l)(9L + \nu_l)\alpha(t-1)\notag\\
		&\quad + 4L\sum_{t=1}^T\sum_{l=1}^2n_l\Gamma_l(L + \nu_l)\sum_{s=1}^{t-1}\theta_l^{t-1-s}\alpha(s-1) \notag\\
		&\quad + 4L\sum_{t=1}^T\sum_{l=1}^2n_l\Gamma_l\theta_l^{t-1}\Lambda_l + \frac{R_1^2}{\alpha(T)}\label{regret_bound_final}
	\end{align}
	where $R_1^2:=\max\{D_{\psi_1}(x_1,x_1^{'}): x_1,x_1^{'}\in X_1\}$ is the diameter of $X_1$.
\end{thm}
The constants of the regret bound depend on the Lipschitz constants, the connectivity of the communication networks and the sampling error of the subgradients. Compared to the regret bound of the centralized online mirror descent with step-sizes $\{\alpha(t)\}_{t\ge 1}$ \cite{s:11}, Theorem \ref{thm1} shows that D-SMD suffers from an additional term $4\sum_{t=1}^T\sum_{l=1}^2(L\frac{n_l(L + \nu_l)\Gamma_l}{1-\theta_l} + n_l\Gamma_l\theta_l^{t-1}\Lambda_l)$, which is caused by the incomplete information of the agents. An immediate corollary is that D-SMD achieves a sublinear regret when $\alpha(t) = t^{-(\frac{1}{2} + \epsilon)}$, where $\epsilon\in [0,1/2)$. Specifically, if we set $\alpha(t) = 1/\sqrt{t}$, it is easy to check that $\bar{R}_1^{(i)}(T) = O(\sqrt{T})$, matching the optimal regret order for convex objectives \cite{h:16}. We formally state the result in the following corollary.
\begin{Col}\label{col1}
	Let the conditions stated in Theorem \ref{thm1} hold. Then, for $\epsilon\in[0,\frac{1}{2})$, Algorithm \ref{alg1} with step-size sequence $\alpha(t) = t^{-(\frac{1}{2} + \epsilon)}$ yields a pseudo regret of order
	\[\bar{R}_1^{(i)}(T) \le O(T^{\frac{1}{2} + \epsilon}).\]
\end{Col}
\begin{proof}
	By exchanging the order of summation, for each $l = 1,2$,
	\begin{align}
		\sum_{t=1}^T\sum_{s=1}^{t-1}\theta_l^{t-1-s}\alpha(s-1)&\le \sum_{t=1}^T\sum_{s=0}^{T-1}\theta_l^s\alpha(t-1)\label{exchange_sum}\\
		&\le \frac{1}{1 - \theta_l}\sum_{t=1}^T\alpha(t-1).\notag
	\end{align}
	Thus, we obtain the result by noting that $\frac{1}{\alpha(T)} = T^{\frac{1}{2} + \epsilon}$ and $\sum_{t=1}^T\alpha(t-1)\le T^{\frac{1}{2} - \epsilon}$.
\end{proof}

Corollary 1 shows that the average pseudo regret of each local agent vanishes as $O(T^{-\frac{1}{2} + \epsilon})$, which indicates that agents can learn the optimal offline strategy merely using local information about the network. 
\subsection{Strongly Convex-Strongly Concave Case}
\cite{ss:07} showed that by using a mirror descent algorithm, the regret bound can be improved to $O(\log T)$ for online optimization problem with generalized strongly convex losses.

In this section, we extend this idea to network stochastic zero-sum games and establish a regret bound of order $O(\log T)$ for our D-SMD. In the following, we give the formal definition of the generalized strongly convex function. 
\begin{Def}[\cite{ss:07}]\label{strong_def}
	A function $f$ is $\eta$-strongly convex over $X$ with respect to a convex and differentiable function $\psi$ if for all $x,y\in X$,
	\[f(x) - f(y) - \langle x-y,\lambda\rangle\ge \eta D_{\psi}(x,y), \forall\lambda\in\partial f(y).\]
\end{Def}
As in \cite{ss:07}, we also need to select a specific step-size sequence depending on the strong convexity coefficient. We formulate the regret bound in this case in Theorem \ref{thm1_2}.
\begin{thm}\label{thm1_2}
	Let the cost function $f_{1,i}(x_1,x_2)$, $i\in\mathcal{V}_1$, be $\eta$-strongly convex in $x_1\in\mathcal{X}_1$ with respect to $\psi_1$ for any $x_2\in\mathcal{X}_2$ and Assumptions \ref{asm1}-\ref{asm3} hold. If $\alpha(t) = \frac{1}{\eta (t+1)}$, then the pseudo regret of D-SMD defined by \eqref{regret_def} is bounded by
	\begin{align}
		\bar{R}_1^{(i)}(T)&\le \sum_{l=1}^2(L + \nu_l)\Big(9L + \nu_l + \frac{4Ln_l\Gamma_l}{1 - \theta_l}\Big)(1 + \log(T))\notag\\
		&\quad + 4L\sum_{l=1}^2\frac{n_l\Gamma_l\Lambda_l\alpha(0)}{1 - \theta_l}.\label{regret_bound_final_1}
	\end{align}
\end{thm}
Comparing this result with Theorem \ref{thm1}, we see that the generalized strong convexity can eliminate the term $\frac{R_1^2}{\alpha(T)}$, which is the main factor restricting the regret order in the convex-concave setting.
\section{Guarantees on Convergence}
In general, \cite{bh:08} proved that a no-regret learning algorithm converges to a coarse correlated equilibrium, which is a relaxation of NE. In this section, we are interested in the convergence of D-SMD to a NE. 
\subsection{Convex-Concave Case}
For a two-person zero-sum game, we can directly derive the convergence rate of the time-averaged strategy profile $(\sum_{t=1}^Tx_1(t)/T,\sum_{t=1}^Tx_2(t)/T)$ from the established regret bound \cite{rs:13b}. However, we cannot apply this result to the two-network stochastic zero-sum game, since each network does not have access to the actual state of its adversarial network. In detail, the difficulty lies in that the pseudo regret is defined from the weighted estimates $u_{l,i}(t)$, and hence, the cumulative cost of $\Sigma_1$ \Big(i.e.,$\mathbb{E}\left[\sum_{t=1}^TU(x_{1,i}(t),u_{2,i}(t))\right]$\Big) is not able to offset the cumulative cost of $\Sigma_2$ \Big(i.e., $\mathbb{E}\left[\sum_{t=1}^T-U(u_{1,i}(t),x_{2,i}(t))\right]$\Big).

We now consider the following time-averaged iterates
\[\hat{x}_{l,i}(t) = \frac{1}{\sum_{s=0}^{t-1}\alpha(s)}\sum_{s=0}^{t-1}\alpha(s)x_{l,i}(s).\quad\text{for}\ t\ge 1, l=1,2\]
In order to measure the approximation quality of the average sequence, we define the gap function
\begin{align}
	&\quad\delta(\hat{x}_{1,i}(t),\hat{x}_{2,j}(t))\notag\\
	&:= \max_{x_2\in X_2}U(\hat{x}_{1,i}(t),x_2) - \min_{x_1\in X_1}U(x_1,\hat{x}_{2,j}(t)).\label{gap_function}
\end{align}
Our goal is to present an expected bound for this gap function. 
\begin{thm}\label{thm5}
	Let $f_{1,i}(\cdot,\cdot)$ be convex-concave and $f_{2,j}(\cdot,\cdot)$ be concave-convex for all $i\in\mathcal{V}_1$, $j\in\mathcal{V}_2$. Suppose that Assumptions \ref{asm1}-\ref{asm3} hold. Let $\{x_{1,i}(s)\}_{0\le s\le t-1}$ and $\{x_{2,j}(s)\}_{0\le s\le t-1}$ be the sequences generated by D-SMD. Then
	\begin{align*}
		\mathbb{E}\left[\delta(\hat{x}_{1,i}(t),\hat{x}_{2,j}(t))\right]&\le \frac{M_1 + M_2\sum_{s=0}^{t-1}\alpha^2(s)}{\sum_{s=0}^{t-1}\alpha(s)}, 
	\end{align*}
	where $M_1:= \sum_{l=1}^2\left(\frac{4Ln_l\Gamma_l\Lambda_l\alpha(0)}{1-\theta_l} + 2R_l^2\right)$ and $M_2 := \sum_{l=1}^2\Big(4L(L + \nu_l)\left(\frac{n_l\Gamma_l}{1 - \theta_l} + 2\right) + (L + \nu_l)^2 + \nu_l^2/2\Big)$.
\end{thm}
We remark that the gap measure in Theorem \ref{thm5} is also used to describe the generalization property of an empirical solution in stochastic saddle point problems \cite{ly:21,zh:21}. It was shown by \cite{ah:21} that $(\hat{x}_{1,i}(t),\hat{x}_{2,j}(t))$ is an $\epsilon$-equilibrium of the two-network stochastic zero-sum game when $\delta(\hat{x}_{1,i}(t),\hat{x}_{2,j}(t))\le \epsilon$. Our result matches the error bound of SMD for saddle point problems \cite{nj:09} with the constants $M_1$ and $M_2$ affected by the structure of the networks.

Since the NE is unique when the cost function is strongly convex-strongly concave, we may transform the expected error bound of Theorem \ref{thm5} into the classical mean-squared error of $(\hat{x}_{1,i}(t),\hat{x}_{2,j}(t))$ in this case.
\begin{Col}\label{col2}
	Suppose that $U(\cdot,\cdot)$ is $\mu$-strongly convex-strongly concave and Assumptions \ref{asm1}-\ref{asm3} hold, and let $\{x_{1,i}(s)\}_{0\le s\le t-1}$ and $\{x_{2,j}(s)\}_{0\le s\le t-1}$ be generated by D-SMD. If $(x_1^{\ast},x_2^{\ast})$ denotes the NE, then
	\begin{align*}
		&\quad\mathbb{E}[\|\hat{x}_{1,i}(t) - x_1^{\ast}\|^2 +\|\hat{x}_{2,j}(t) - x_2^{\ast}\|^2]\\
		&\le \frac{2}{\mu}\frac{M_1 + M_2\sum_{s=0}^{t-1}\alpha^2(s)}{\sum_{s=0}^{t-1}\alpha(s)}, 
	\end{align*}
	where $M_1$ and $M_2$ are as defined in Theorem \ref{thm5}.
\end{Col}
This corollary implies that when $\{\alpha(t)\}_{t\ge 0}$ satisfies $\sum_{t=0}^{\infty}\alpha(t) = \infty$ and $\sum_{t=0}^{\infty}\alpha^2(t) < \infty$, $(\hat{x}_{1,i}(t),\hat{x}_{2,j}(t))$ converges in mean square to the unique NE for all $i\in\mathcal{V}_1$, $j\in\mathcal{V}_2$. 

\subsection{Strictly Convex-Strictly Concave Case}
We now study the almost sure convergence of the strategy profile generated by D-SMD in the strictly convex-strictly concave setting. This usually requires an analysis tool quite different from that of the time-averaged sequence \cite{mz:19}.
\begin{thm}\label{thm2}
	Let $U(\cdot,\cdot)$ be strictly convex-strictly concave and Assumptions \ref{asm1}-\ref{asm3} hold. If the step-size sequence satisfies $\sum_{t=1}^{\infty}\alpha(t) = \infty$ and $\sum_{t=1}^{\infty}\alpha^2(t) < \infty$, then D-SMD almost surely converges to the unique NE, denoted by $x^{\ast} = (x_1^{\ast},x_2^{\ast})$, i.e., with probability $1$,
	\[\lim_{t\to\infty}x_{1,i}(t) = x_1^{\ast},\quad \lim_{t\to\infty}x_{2,j}(t) = x_2^{\ast},\ \forall i\in\mathcal{V}_1, j\in\mathcal{V}_2.\]
\end{thm}
We conclude from Corollary \ref{col1} and Theorem \ref{thm2} that for strongly convex-strongly concave network stochastic zero-sum games, D-SMD is a no-regret learning process that converges to the NE when $\sum_{t=0}^{\infty}\alpha(t) = \infty$ and $\sum_{t=0}^{\infty}\alpha^2(t) < \infty$. Moreover, this result generalizes Theorem 6 of \cite{srivastava2013distributed}, which is only applicable to a deterministic distributed saddle point problem. 

\section{Numerical Results}
In this section, we conduct numerical experiments for a network version of the two-person stochastic matrix games \cite{zh:21} to evaluate the performance of the proposed D-SMD algorithm in convex-concave and strongly convex-strongly concave cases.

For the above two cases, we set the number of players in both networks to be $N = 12$ and let each player have $K=20$ actions to choose from. We focus on studying the influence of the step-size sequence and the network topology on the regret bound and convergence of D-SMD. Specifically, we fix the structure of $\Sigma_2$ and consider three types of graphs with different degree of connectivity (Cycle graph $<$ Random graph $<$ Complete graph) for $\Sigma_1$ in our simulations.
\begin{itemize}
	\item {\it Cycle graph} has a single cycle and each node has exactly two immediate neighbors.
	\item {\it Random graph} is constructed by connecting nodes randomly and each edge is included in the graph with probability $0.7$ independent from every other edge.
	\item {\it Complete graph} is constructed by connecting all of the node pairs.                                                      
\end{itemize} 

\subsection{Convex-Concave Case}
Consider the following network stochastic zero-sum game
\[\min_{x_1\in\Delta_K}\max_{x_2\in\Delta_K}U(x_1,x_2) :=\frac{1}{N}\sum_{i=1}^Nx_1^T\mathbb{E}_{\xi}[A_{\xi}^i]x_2\]
where $x_1$ and $x_2$ are the mixed strategies of players in $\Sigma_1$ and $\Sigma_2$, respectively, which belong to the simplex $\Delta_K:=\{z\in\mathbb{R}^K: z\ge 0, \textbf{1}^Tz = 1\}$. $A_{\xi}^i$ is the stochastic cost matrix of player $i\in\Sigma_1$. Let $\{x_{1,i}(t)\}_{t\ge 0}$ and $\{x_{2,i}(t)\}_{t\ge 0}$ be the outputs of D-SMD. Recalling the gap function \eqref{gap_function}, we use its average with respect to all players, defined as $\bar{\delta}(t):= \frac{1}{N^2}\sum_{i=1}^N\sum_{j=1}^N\delta(\hat{x}_{1,i}(t),\hat{x}_{2,j}(t))$, to demonstrate the convergence of D-SMD.

In the first experiment, we run D-SMD from $t=1$ to $t=500$ for different learning rates $\alpha(t) = t^{-\frac{1}{2}}, t^{-\frac{2}{3}}, t^{-\frac{3}{4}}$ and estimate the expected average gap $\mathbb{E}[\bar{\delta}(t)]$ and the time-averaged pseudo regret $\bar{R}_1^{(i)}(t)/t$ by averaging across $50$ sample paths. The empirical results are shown in Figure \ref{fig1}, which shows that D-SMD can achieve sublinear regret bound and the slower learning rate yields a better regret rate and a faster convergence rate. These are consistent with our theoretical results in the convex-concave case. 

In the second experiment, we use the learning rate $\alpha(t) = 1/\sqrt{t}$ and compare the expected average gap and average pseudo regret with different network topologies in Figure \ref{fig2}, which demonstrates that the network topology has only a slight influence on the regret rate and convergence rate. 

\begin{figure}[h]
	\centering
	\subfigure[]{\includegraphics[width = 0.48\columnwidth]{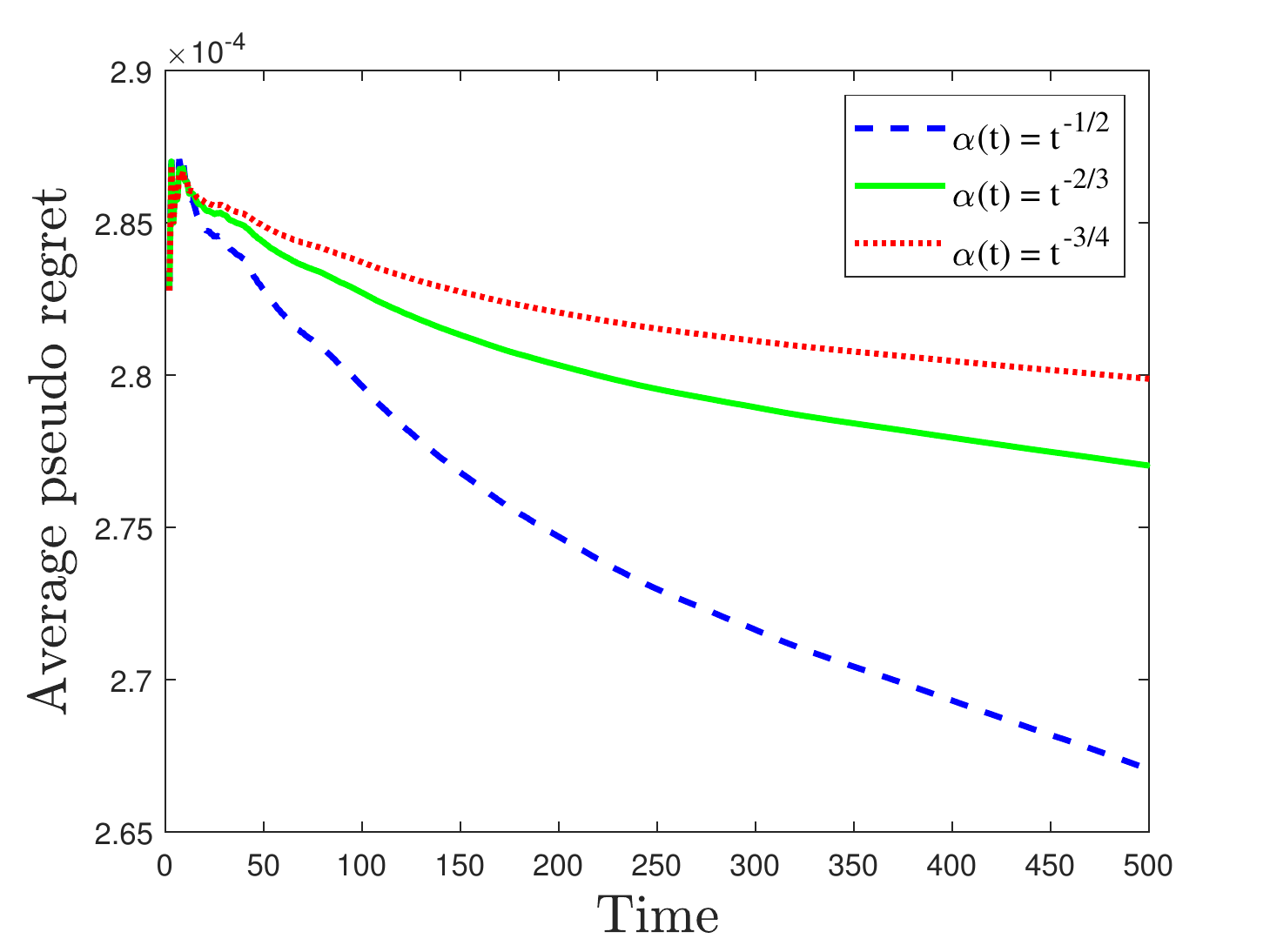}}
	\subfigure[]{\includegraphics[width = 0.48\columnwidth]{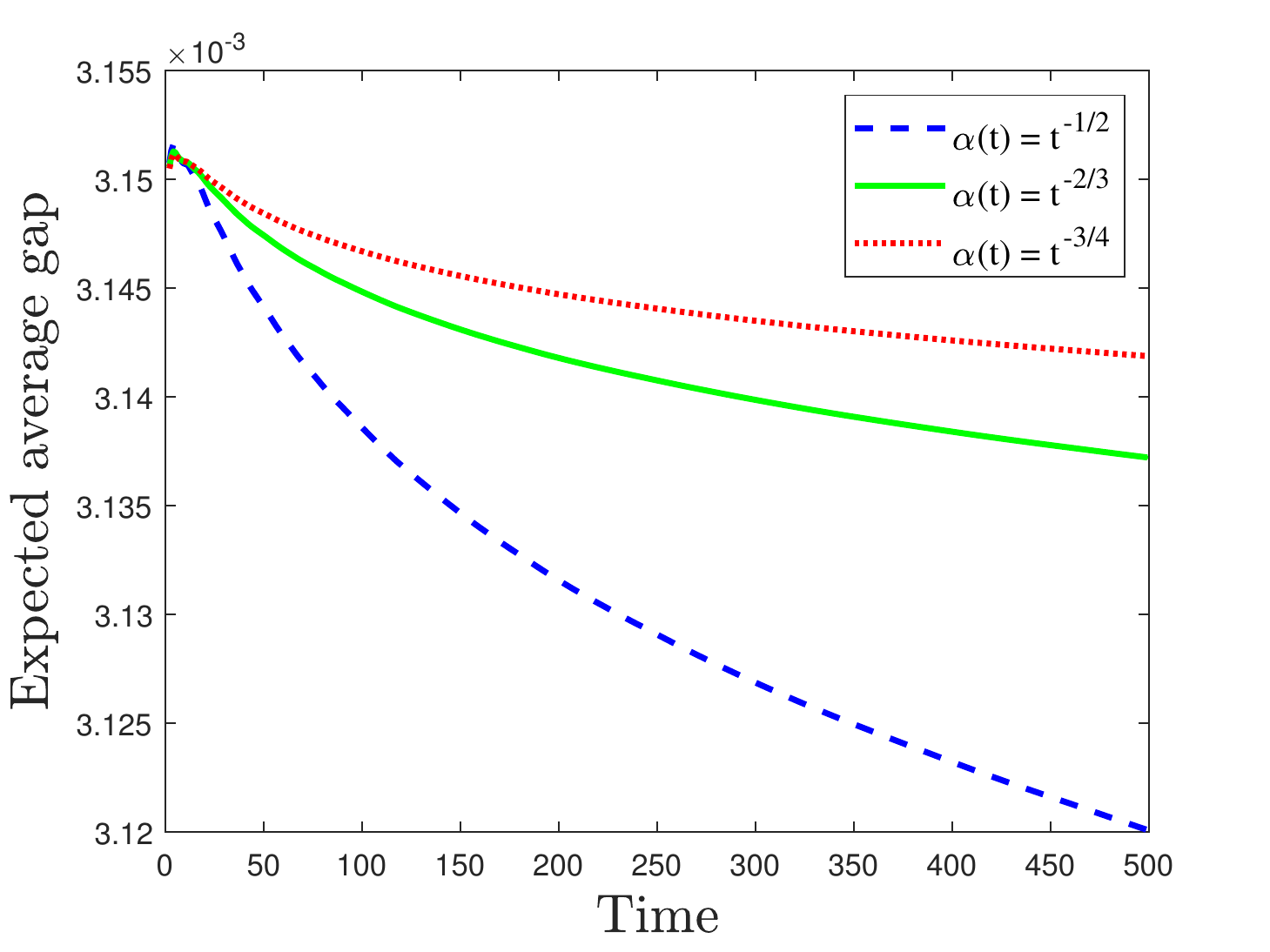}}
	\caption{Average pseudo regret and expected average gap and  of D-SMD under different learning rates in the convex-concave case}
	\label{fig1}
\end{figure}
\begin{figure}[h]
	\centering
	\subfigure[]{\includegraphics[width = 0.48\columnwidth]{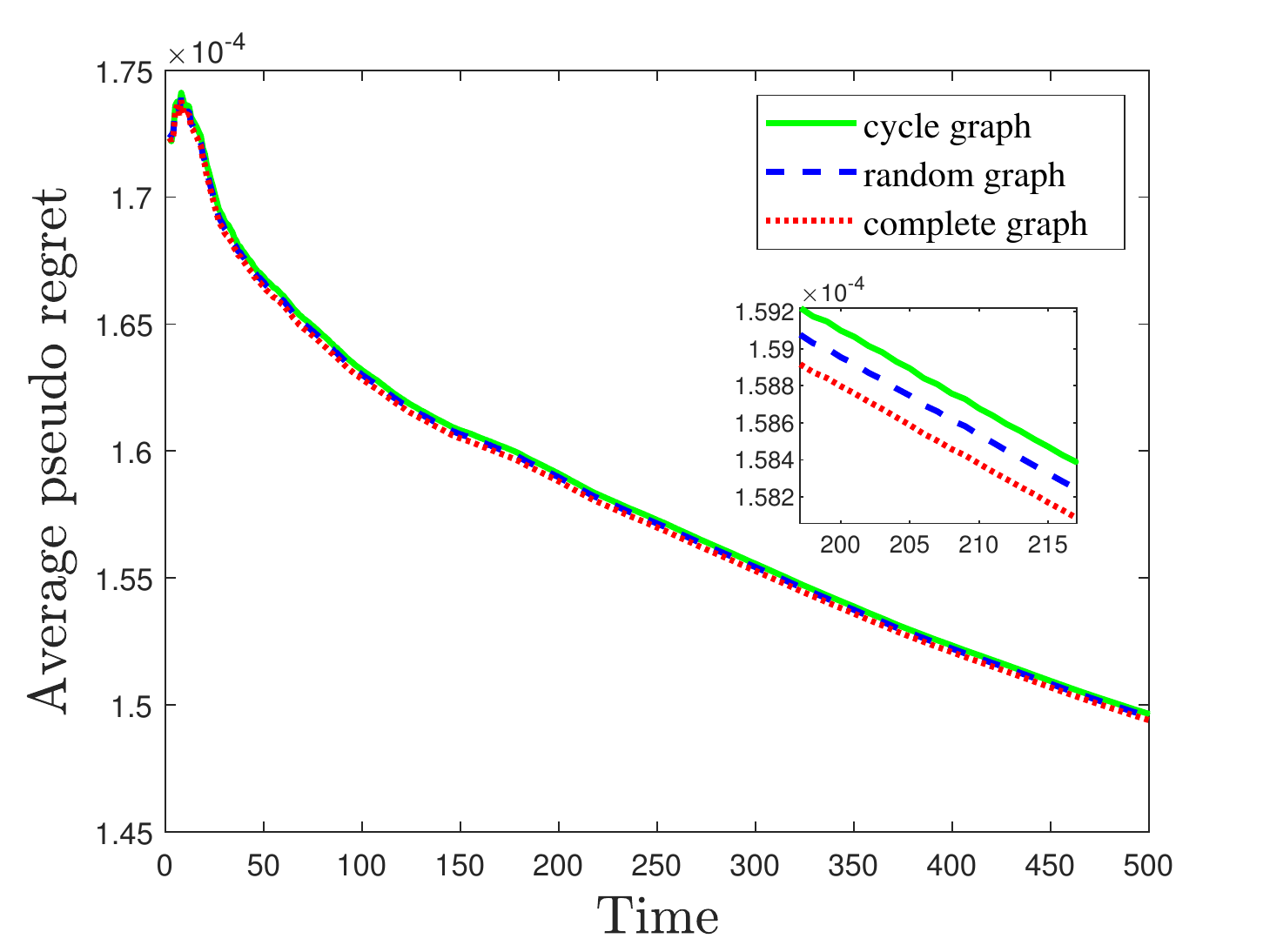}}
	\subfigure[]{\includegraphics[width = 0.48\columnwidth]{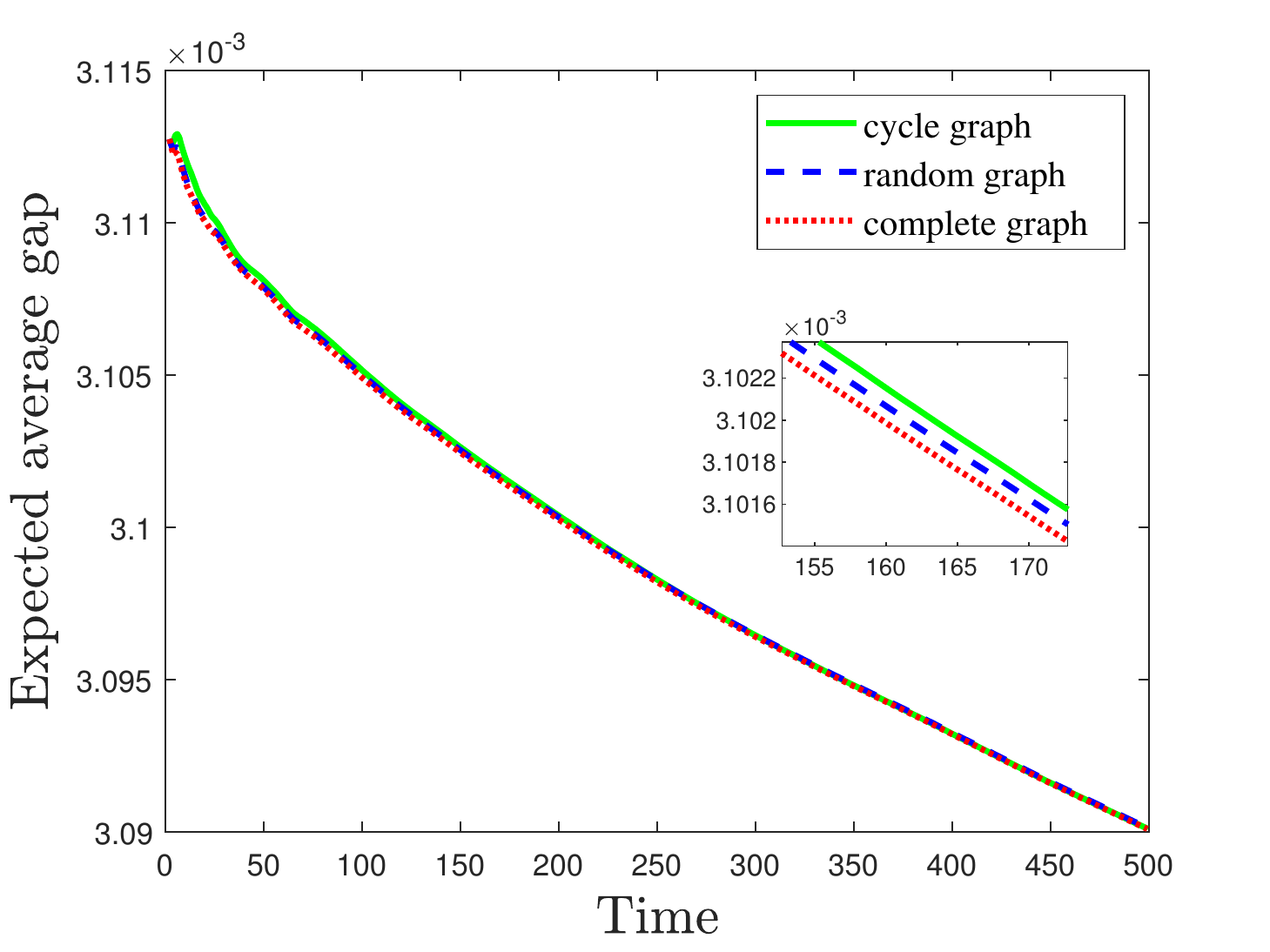}}
	\caption{Average pseudo regret and expected average gap and  of D-SMD under different network topologies in the convex-concave case}
	\label{fig2}
\end{figure}
\subsection{Strongly Convex-Strongly Concave Case}
To investigate the performance of D-SMD in the strongly convex-strongly concave case, we consider the regularized network stochastic zero-sum game with cost function defined as follows
\begin{align*}
	U(x_1,x_2)
	&:= \sum_{p=1}^Kx_1^{(p)}\log x_1^{(p)} + \frac{1}{N}\sum_{i=1}^Nx_1^T\mathbb{E}_{\xi}[A_{\xi}^i]x_2\\
	&\quad - \sum_{p=1}^Kx_2^{(p)}\log x_2^{(p)}.
\end{align*}  
Notice that $U(x_1,x_2)$ is $1$-strongly convex-strongly concave with respect to the regularizer $\psi(x) = \sum_{p=1}^Kx^{(p)}\log x^{(p)}$. We use a dual averaging algorithm proposed in \cite{m:19} to compute the Nash equilibrium, denoted by $(x_1^{\ast},x_2^{\ast})$. Due to the uniqueness of the NE in the strongly convex-strongly concave case, we consider the average absolute error $\bar{\delta}'(t):=\frac{1}{N}\sum_{i=1}^N\|x_{1,i}(t) - x_1^{\ast}\| + \|x_{2,i}(t) - x_2^{\ast}\|$ to illustrate the convergence of the sequences $\{(x_{1,i}(t),x_{2,i}(t))\}_{t\ge 0}$. The time-averaged pseudo regret $\bar{R}_1^{(i)}(t)/t$ and the expected absolute error $\mathbb{E}[\bar{\delta}'(t)]$ under different learning rates are plotted in Figure 3, and the performance under different network topologies with learning rate $\alpha(t) = 1/t$ are displayed in Figure 4. D-SMD produces a better regret rate in the strongly convex-strongly concave case and the influence of network topology and learning rate is similar to the convex-conacve case. 
\begin{figure}[h]
	\centering
	\subfigure[]{\includegraphics[width = 0.48\columnwidth]{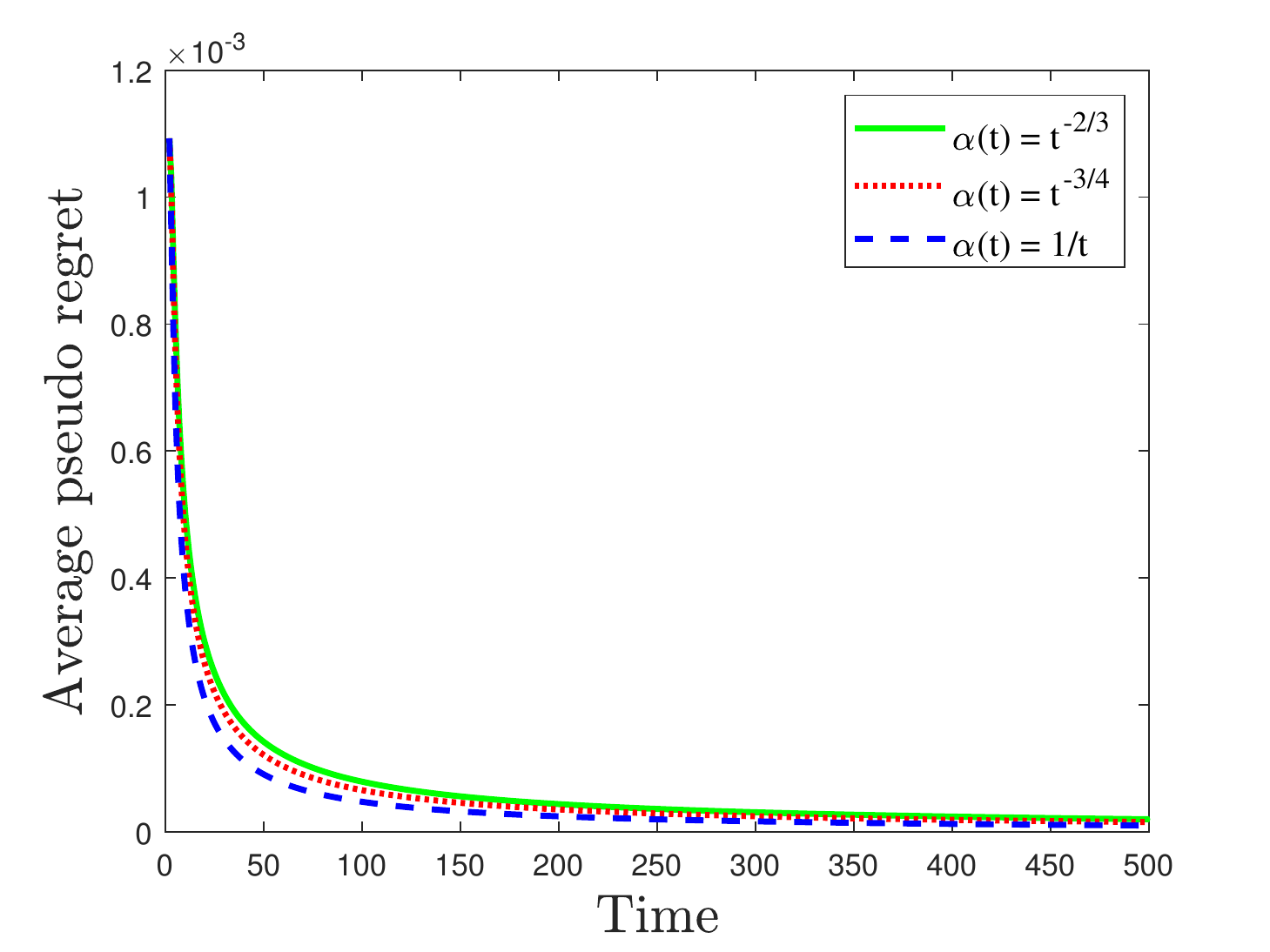}}
	\subfigure[]{\includegraphics[width = 0.48\columnwidth]{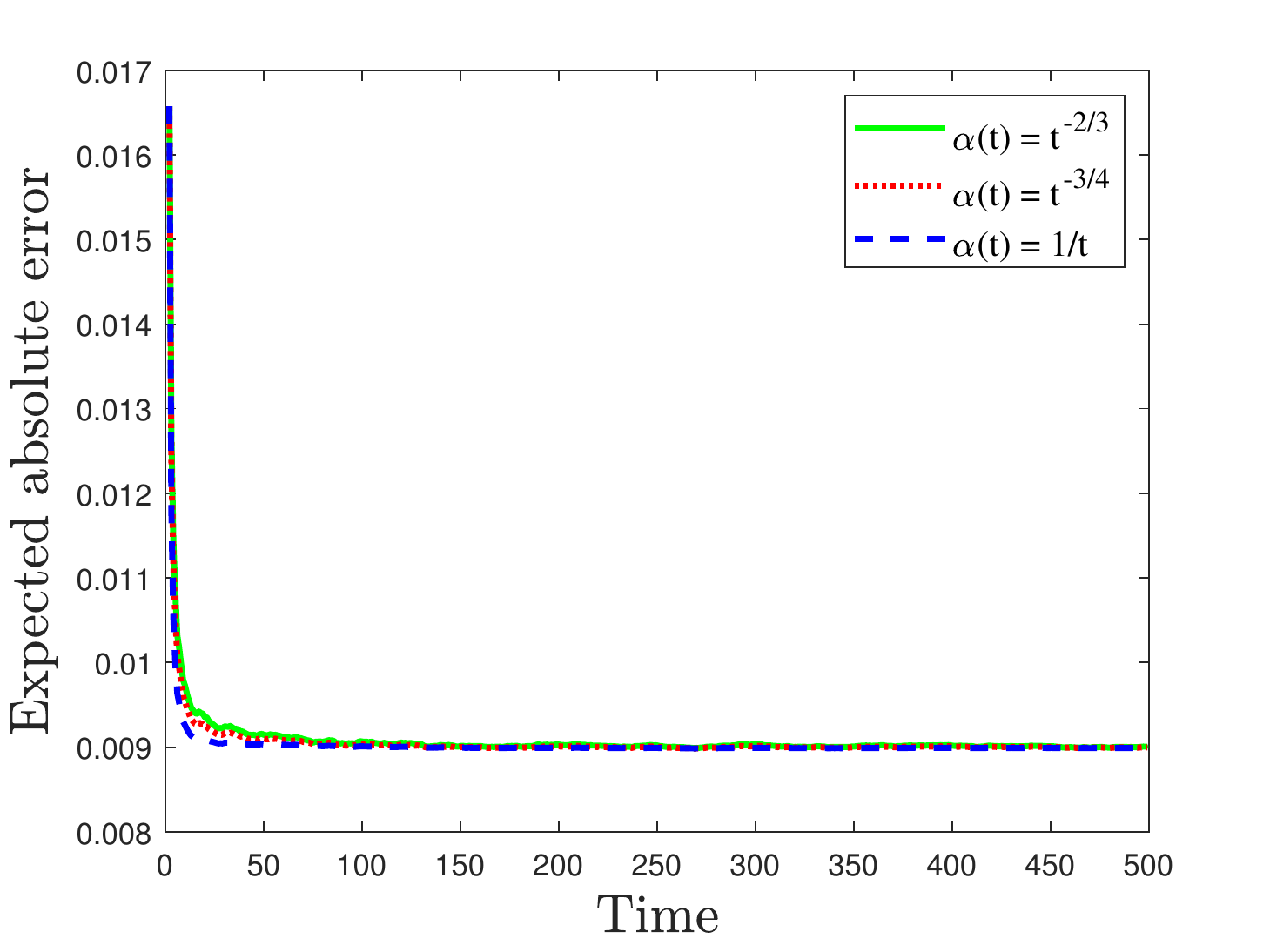}}
	\caption{Average pseudo regret and expected absolute error of D-SMD under different learning rates in the strongly convex-strongly concave case}
	\label{fig3}
\end{figure}
\begin{figure}[h]
	\centering
	\subfigure[]{\includegraphics[width = 0.48\columnwidth]{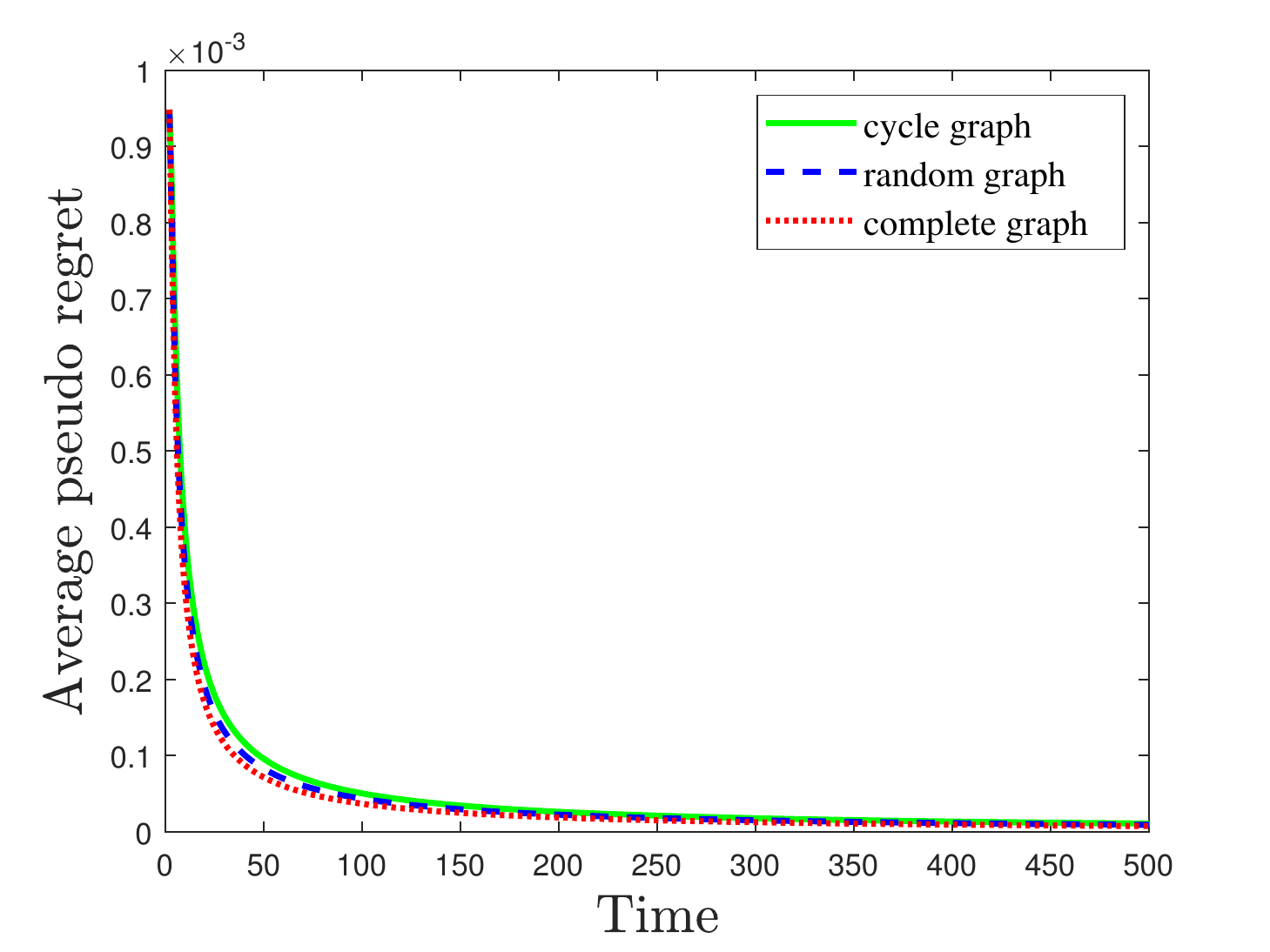}}
	\subfigure[]{\includegraphics[width = 0.48\columnwidth]{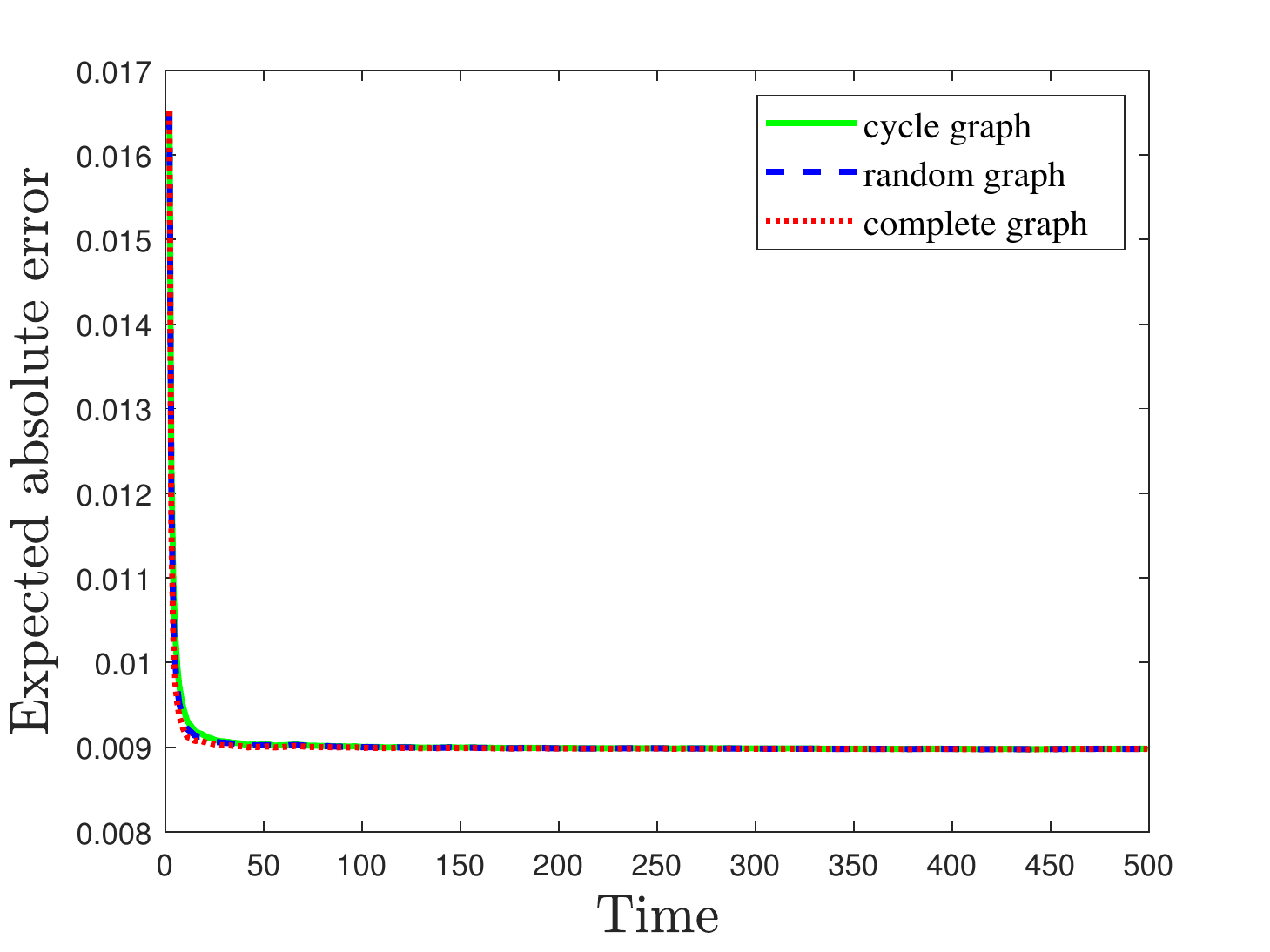}}
	\caption{Average pseudo regret and expected absolute error of D-SMD under different network topologies in the strongly convex-strongly concave case}
	\label{fig4}
\end{figure}    
\section{Conclusion and Future Work}
In this paper, we proposed distributed stochastic mirror descent (D-SMD) to extend no-regret learning in two-person zero-sum games to network stochastic zero-sum games. In contrast to the previous works on Nash equilibrium seeking for network zero-sum games, we not only derived the convergence of D-SMD to the set of Nash equilibria, but also established regret bounds of D-SMD for convex-concave and strongly convex-strongly concave costs. The theoretical results were empirically verified by experiments on solving network stochastic matrix games.

It is of interest to study the convergence rate of the actual iterates of D-SMD in the strongly convex-strongly concave case. In addition, another interesting topic is to develop an optimistic variant of our algorithm as in \cite{ddk:11} to improve the regret rate and obtain the last-iteration convergence in merely convex-concave case.
\section*{Appendix}
\begin{lem}\cite[Lemma B.2, Proposition B.3]{ml:19}\label{lem1}
	Let $\psi$ be a continuously differentiable $\sigma$-strongly convex function on $\mathcal{X}$. Then, for all $x,y,z\in\mathcal{X}$, the Bregman divergence defined by \eqref{breg_def} satisfies
	\begin{align}
		D_{\psi}(y,x) - D_{\psi}(y,z) - D_{\psi}(z,x) &= \langle\nabla\psi(z) - \nabla\psi(x),y-z\rangle,\label{breg_prop1}\\
		D_{\psi}(x,y)&\ge \frac{\sigma}{2}\|x - y\|^2\label{breg_prop2}
	\end{align}
	Moreover, let $x^{+} = P_x(v)$ for $v\in\mathcal{X}^{\ast}$, where $\mathcal{X}^{\ast}$ is the dual space of $\mathcal{X}$ and $P_x(v) := \arg\min_{x'\in \mathcal{X}}\{\langle v,x - x'\rangle + D_{\psi}(x',x)\}$. Then
	\begin{equation}
		D_{\psi}(y,x^{+}) \le D_{\psi}(y,x) + \langle v,x - y\rangle + \frac{1}{2\sigma}\|v\|_{\ast}^2,
	\end{equation}
	where $\|v\|_{\ast}:= \sup\{\langle v,x\rangle: x\in\mathcal{X},\|x\|\le 1\}$ denotes the dual norm on $\mathcal{X}$.
\end{lem}
\begin{lem}\cite{rs:71}\label{lem3}
	Let $\{X_t\}$, $\{Y_t\}$ and $\{Z_t\}$ be sequences of non-negative random variables with $\sum_{t=0}^{\infty}Z_t < \infty$ almost surely and let $\{\mathcal{F}_t\}$ be a filtration such that $\mathcal{F}_t\subset\mathcal{F}_{t+1}$. If $X_t$, $Y_t$, $Z_t$ are adapted to $\{\mathcal{F}_t\}$ and 
	\[\mathbb{E}[Y_{t+1}\mid\mathcal{F}_t] \le Y_t - X_t + Z_t,\] 
	then, almost surely, $\sum_{t=0}^{\infty}X_t < \infty$ and $Y_t$ converges to a non-negative random variable $Y$. 
\end{lem}
\begin{lem}\label{lem2}
	Let Assumptions \ref{asm1}-\ref{asm2} hold. Suppose that $x_{l,i}(t)$, $v_{l,i}(t)$ and $u_{l,i}(t)$ are generated by Algorithm \ref{alg1}. Then, for each $l\in\{1,2\}$ and $i\in\mathcal{V}_l$,
	\begin{align}
		E[\|x_{l,i}(t) - \bar{x}_{l}(t)\|] &\le H_l(t),\label{lem_bound_2}\\
		E[\|\bar{x}_{l}(t) - v_{l,i}(t)\|] &\le H_l(t),\label{lem_bound_3}\\
		E[\|\bar{x}_{l}(t) - u_{l,i}(t)\|] &\le H_l(t),\label{lem_bound_4}
	\end{align}
	where
	\begin{align}
		H_l(t) &= n_l\Gamma_l\theta_l^{t-1}\Lambda_l + 2(L + \nu_l)\alpha(t-1) + n_l\Gamma_l(L + \nu_l)\sum_{s=1}^{t-1}\theta_l^{t-1-s}\alpha(s-1),\label{H_l}
	\end{align}
	and $\Lambda_l\triangleq \max_{i\in\mathcal{V}_l}\|x_{l,i}(0)\|$.
\end{lem}
\begin{proof}
By the definition of $v_{l,i}(t)$, we write the iterates as follows
\begin{align*}
	x_{l,i}(t) &= v_{l,i}(t-1) - (v_{l,i}(t-1) - x_{l,i}(t))\notag\\
	&= \sum_{j=1}^{n_l}[\Phi_l(t-1,0)]_{ij}x_{l,j}(0) + \sum_{s=1}^{t-1}\sum_{j=1}^{n_l}[\Phi_l(t-1,s)]_{ij}d_{l,j}(s-1) + d_{l,i}(t-1),
\end{align*}
where $d_{l,i}(t-1) = x_{l,i}(t) - v_{l,i}(t-1)$. By using the doubly stochastic property of $W_l(t)$, we derive
\begin{equation*}
	\bar{x}_l(t) = \frac{1}{n_l}\sum_{j=1}^{n_l}x_{l,j}(0) + \frac{1}{n_l}\sum_{s=1}^t\sum_{j=1}^{n_l}d_{l,j}(s-1).
\end{equation*}
Therefore,
\begin{align}
	\|x_{l,i}(t) - \bar{x}_l(t)\|&\le \sum_{j=1}^{n_l}\|[\Phi_l(t-1,0)]_{ij} - \frac{1}{n_l}\|\|x_{l,j}(0)\|\notag\\
	&\quad + \sum_{s=1}^{t-1}\sum_{j=1}^{n_l}\|[\Phi_l(t-1,s)]_{ij} - \frac{1}{n_l}\|\|d_{l,j}(s-1)\| + \|\frac{1}{n_l}\sum_{j=1}^{n_l}d_{l,j}(t-1) - d_{l,i}(t-1)\|.\label{bound_3}
\end{align}
Thus, we only need to bound the term $\|d_{l,i}(t)\|$. Recalling the definition of $x_{l,i}(t+1)$ and \eqref{prox_mapping}, we have
\begin{equation}\label{def_x}
	x_{l,i}(t+1) = \arg\min_{x'\in X_l}\{\langle \alpha(t)\hat{g}_{l,i}(t),x' - v_{l,i}(t)\rangle + D_{\psi_l}(x',v_{l,i}(t))\}.
\end{equation}
From the first-order optimality condition, we derive that for all $x_l\in X_l$,
\begin{equation}\label{proj}
	\langle \nabla \psi_l(x_{l,i}(t+1)) - \nabla \psi_l(v_{l,i}(t)) + \alpha(t)\hat{g}_{l,i}(t), x_{l,i}(t+1) - x_l\rangle \le 0.
\end{equation}
Setting $x_l = v_{l,i}(t)$ implies
\begin{align*}
	\alpha(t)\|\hat{g}_{l,i}(t)\|_{\ast}\|d_{l,i}(t)\| &\ge \langle\alpha(t)\hat{g}_{l,i}(t), v_{l,i}(t) - x_{l,i}(t+1)\rangle\\
	&\ge \langle \nabla \psi_l(x_{l,i}(t+1)) - \nabla \psi_l(v_{l,i}(t)), x_{l,i}(t+1) - v_{l,i}(t)\rangle\\
	&\ge \|d_{l,i}(t)\|^2,
\end{align*}
where the last inequality follows by the strong convexity of $\psi_l$. Therefore, \begin{equation}\label{error_bound}
	\|d_{l,i}(t)\|\le \alpha(t)\|\hat{g}_{l,i}(t)\|_{\ast}.
\end{equation}
It follows from Assumption \ref{asm1} that $\|g_{l,i}(t)\|_{\ast}\le L$. By H\"{o}lder's inequality and the bounded second moment condition of Assumption \ref{asm2}, we further achieve
\begin{equation}\label{sample_gradient_bound}
	\mathbb{E}[\|\hat{g}_{l,i}(t)\|_{\ast}^2|\mathcal{F}_t]\le (L + \nu_l)^2.
\end{equation}
Note that $\sqrt{x}$ is a concave function. Using Jensen's inequality,
\[\mathbb{E}[\|\hat{g}_{l,i}(t)\|_{\ast}|\mathcal{F}_t]\le \sqrt{\mathbb{E}[\|\hat{g}_{l,i}(t)\|_{\ast}^2|\mathcal{F}_t]}\le L + \nu_l.\]
According to the iterated expectation rule, $\mathbb{E}[\|\hat{g}_{l,i}(t)\|_{\ast}]\le L + \nu_l$. This together with \eqref{error_bound} produces 
$$\mathbb{E}[\|d_{l,i}(t)\|]\le (L + \nu_l)\alpha(t).$$
Then, by Lemma \ref{lem_graph} and taking the expectation in \eqref{bound_3}, we derive \eqref{lem_bound_2}.
Furthermore, by the convexity of $\|\cdot\|$ and $\sum_{j=1}^{n_l}w_{l,ij}(t) = 1$, we obtain
\begin{align*}
	\mathbb{E}[\|v_{l,i}(t) - \bar{x}_{l}(t)\|]
	& = \mathbb{E}\left[\left\|\sum_{j=1}^{n_l}w_{l,ij}(t)x_{l,j}(t) - \bar{x}_{l}(t)\right\|\right] \le \sum_{j=1}^{n_l}w_{l,ij}(t)\mathbb{E}[\|x_{l,j}(t) - \bar{x}_{l}(t)\|]\le H_l(t).
\end{align*}
Thus, \eqref{lem_bound_3} holds. In a similar way, by using $\sum_{j=1}^{n_l}w_{12,ij}(t) = 1$, we obtain \eqref{lem_bound_4}.
\end{proof}

\noindent{\bf Proof of Theorem 2}. By \eqref{proj} and using Lemma \ref{lem1}, 
we obtain that, for all $x_1\in\mathcal{X}_1$, 
\begin{align}
	\langle \alpha(t)\hat{g}_{1,i}(t),x_{1,i}(t+1) - x_1\rangle &\le \langle \nabla\psi_1(v_{1,i}(t)) - \nabla\psi_1(x_{1,i}(t+1)),x_{1,i}(t+1) - x_1\rangle\notag\\
	&= D_{\psi_1}(x_1, v_{1,i}(t)) - D_{\psi_1}(x_1, x_{1,i}(t+1)) - D_{\psi_1}(x_{1,i}(t+1), v_{1,i}(t))\notag\\
	&\le D_{\psi_1}(x_1, v_{1,i}(t)) - D_{\psi_1}(x_1, x_{1,i}(t+1)).\label{thm2-1}
\end{align}
Note by Assumption \ref{asm3} that 
\begin{align*}
	\sum_{t=1}^T\frac{1}{\alpha(t)}\sum_{i=1}^{n_1}D_{\psi_1}(x_1,v_{1,i}(t+1))&\le \sum_{t=1}^T\frac{1}{\alpha(t)}\sum_{i=1}^{n_1}\sum_{j=1}^{n_1}w_{1,ij}(t)D_{\psi_1}(x_1,x_{1,j}(t+1))\\
	& = \sum_{t=1}^T\frac{1}{\alpha(t)}\sum_{i=1}^{n_1}D_{\psi_1}(x_1,x_{1,i}(t+1)).
\end{align*}
Thus, dividing $\alpha(t)$ from both sides of \eqref{thm2-1} and taking a summation for $i=1,\dots,n_1$ and $t=1,\dots,T$, we derive
\begin{align}
	&\quad\sum_{t=1}^T\sum_{i=1}^{n_1}\langle \hat{g}_{1,i}(t), x_{1,i}(t+1) - x_1\rangle \notag\\
	&\le \sum_{i=1}^{n_1}\sum_{t=1}^T\frac{1}{\alpha(t)}[D_{\psi_1}(x_1, v_{1,i}(t)) - D_{\psi_1}(x_1, v_{1,i}(t+1))]  \label{inner_bound_1}\\
	&\le \sum_{i=1}^{n_1}\left[\left(\frac{1}{\alpha(1)}\right)D_{\psi_1}(x_1, v_{1,i}(1)) + \sum_{t=2}^TD_{\psi_1}(x_1, v_{1,i}(t))\left(\frac{1}{\alpha(t)} - \frac{1}{\alpha(t-1)}\right)\right]\label{bound_2}\\
	&\le \frac{n_1R_1^2}{\alpha(T)},\notag
\end{align}
where the last inequality follows from the definition of $R_1^2$ and the non-increasing of $\alpha(t)$. This together with \eqref{error_bound} and the definition $d_{l,i}(t) = x_{l,i}(t+1) - v_{l,i}(t)$ yields
\begin{align}
	&\quad\mathbb{E}\left[\sum_{t=1}^T\sum_{i=1}^{n_1}\langle \hat{g}_{1,i}(t), v_{1,i}(t) - x_1\rangle\right]\notag\\
	&= \sum_{t=1}^T\sum_{i=1}^{n_1}\mathbb{E}[\langle \hat{g}_{1,i}(t), x_{1,i}(t+1) - x_1\rangle] + \sum_{t=1}^T\sum_{i=1}^{n_1}\mathbb{E}[\langle \hat{g}_{1,i}(t), v_{1,i}(t) - x_{1,i}(t+1)\rangle]\notag\\
	&\le \frac{n_1R_1^2}{\alpha(T)} + \sum_{t=1}^T\sum_{i=1}^{n_1}\alpha(t)\mathbb{E}[\|\hat{g}_{1,i}(t)\|_{\ast}^2]\le \frac{n_1R_1^2}{\alpha(T)} + n_1(L+\nu_1)^2\sum_{t=1}^T\alpha(t) ,\label{upper_bound_1}
\end{align}
where the last inequality follows from \eqref{sample_gradient_bound}. Since $v_{1,i}(t)$ is adapted to $\mathcal{F}_t$, by Assumption \ref{asm2},
\begin{equation}\label{condition_unbias}
	\mathbb{E}[\langle g_{1,i}(t) - \hat{g}_{1,i}(t),v_{1,i}(t) - x_1\rangle|\mathcal{F}_{t}] = 0.
\end{equation}
Therefore,
\begin{equation}\label{unbias}
	\mathbb{E}[\langle g_{1,i}(t) - \hat{g}_{1,i}(t),v_{1,i}(t) - x_1\rangle] = 0.
\end{equation}
As a combination of \eqref{upper_bound_1} and \eqref{unbias}, we conclude
\begin{equation}\label{upper_bound}
	\mathbb{E}\left[\sum_{t=1}^T\sum_{i=1}^{n_1}\langle g_{1,i}(t), v_{1,i}(t) - x_1\rangle\right] \le \frac{n_1R_1^2}{\alpha(T)} + n_1(L+\nu_1)^2\sum_{t=1}^T\alpha(t).
\end{equation}

Next, we establish a lower bound for $\sum_{i=1}^{n_1}\langle g_{1,i}(t), v_{1,i}(t) - x_1\rangle$. Due to the convexity of $f_{1,i}(\cdot,\cdot)$ with respect to the first element, 
\begin{align}
	\sum_{i=1}^{n_1}\langle g_{1,i}(t), v_{1,i}(t) - x_1\rangle &\ge \sum_{i=1}^{n_1}[f_{1,i}(v_{1,i}(t),u_{2,i}(t)) - f_{1,i}(x_1,u_{2,i}(t))].\label{lower_bound_1}
\end{align}
On the other hand, by adding and subtracting some terms, we get
\begin{align}
	&\quad n_1(U(\bar{x}_1(t),\bar{x}_2(t)) - U(x_1,\bar{x}_2(t)))\notag\\
	&= \sum_{i=1}^{n_1}[f_{1,i}(\bar{x}_1(t),\bar{x}_2(t)) - f_{1,i}(x_{1,i}(t),\bar{x}_2(t)) + f_{1,i}(x_{1,i}(t),\bar{x}_2(t)) - f_{1,i}(v_{1,i}(t),\bar{x}_2(t))\notag\\
	&\quad + f_{1,i}(v_{1,i}(t),\bar{x}_{2}(t)) - f_{1,i}(v_{1,i}(t),u_{2,i}(t)) + f_{1,i}(v_{1,i}(t),u_{2,i}(t)) - f_{1,i}(x_1,u_{2,i}(t))\notag\\
	&\quad + f_{1,i}(x_1,u_{2,i}(t)) - f_{1,i}(x_1,\bar{x}_2(t))]\notag\\
	&\le \sum_{i=1}^{n_1}[L(\|x_{1,i}(t) - \bar{x}_1(t)\| + \|x_{1,i}(t) - v_{1,i}(t)\|) + 2L\|u_{2,i}(t) - \bar{x}_2(t)\|\notag\\
	&\quad + f_{1,i}(v_{1,i}(t),u_{2,i}(t)) - f_{1,i}(x_1,u_{2,i}(t))],\label{add_substract}
\end{align}
where the last inequality follows from the Lipschitz continuity of $f_{1,i}$. Plugging the above inequality to \eqref{lower_bound_1}, we derive
\begin{align}
	\sum_{i=1}^{n_1}\langle g_{1,i}(t), v_{1,i}(t) - x_1\rangle &\ge n_1(U(\bar{x}_1(t),\bar{x}_2(t)) - U(x_1,\bar{x}_2(t)))\notag\\
	&\quad - \sum_{i=1}^{n_1}[L(\|x_{1,i}(t) - \bar{x}_1(t)\| + \|x_{1,i}(t) - v_{1,i}(t)\|) + 2L\|u_{2,i}(t) - \bar{x}_2(t)\|].\label{lower_bound_2}
\end{align}
It remains to connect this lower bound and $\bar{R}_1^{(i)}(T)$. Notice that
\begin{align}
	&\quad U(\bar{x}_1(t),\bar{x}_2(t)) - U(x_1,\bar{x}_2(t))\notag\\
	&= U(\bar{x}_1(t),\bar{x}_2(t)) - U(\bar{x}_1(t),u_{2,i}(t)) + U(\bar{x}_1(t),u_{2,i}(t)) - U(x_{1,i}(t),u_{2,i}(t))\notag\\
	&\quad + U(x_{1,i}(t),u_{2,i}(t)) - U(x_1,u_{2,i}(t)) + U(x_1,u_{2,i}(t)) - U(x_1,\bar{x}_2(t))\notag\\
	&\ge U(x_{1,i}(t),u_{2,i}(t)) - U(x_1,u_{2,i}(t)) - L\|x_{1,i}(t) - \bar{x}_1(t)\|\notag\\
	&\quad - 2L\|u_{2,i}(t) - \bar{x}_2(t)\|.\label{lower_bound_3}
\end{align}
Recall from the definition of $\bar{R}_1^{(i)}(T)$ that
\begin{align*}
	\max_{x_1\in X_1}\mathbb{E}\left[\sum_{t=1}^T(U(x_{1,i}(t),u_{2,i}(t)) - U(x_1,u_{2,i}(t)))\right]&= \bar{R}_1^{(i)}(T).
\end{align*}
By taking expectation on both sides of \eqref{lower_bound_2}-\eqref{lower_bound_3} and making a summation from $t=1$ to $t=T$, we obtain
\begin{align}
	\max_{x_1\in X_1}\mathbb{E}\left[\sum_{t=1}^T\sum_{i=1}^{n_1}\langle g_{1,i}(t), v_{1,i}(t) - x_1\rangle\right] &\ge n_1\bar{R}_1^{(i)}(T) - n_1L\sum_{t=1}^T\mathbb{E}[\|x_{1,i}(t) - \bar{x}_1(t)\| + 2\|u_{2,i}(t) - \bar{x}_2(t)\|]\notag\\
	&\quad -L\sum_{t=1}^T\sum_{i=1}^{n_1}\mathbb{E}\Big[\|x_{1,i}(t) - \bar{x}_1(t)\| + \|x_{1,i}(t) - v_{1,i}(t)\|\notag\\
	&\quad + 2\|u_{2,i}(t) - \bar{x}_2(t)\|\Big]\label{lower_bound}
\end{align}
Note by Lemma \ref{lem2} and the elementary inequality $\|a + b\|\le \|a\| + \|b\|$ that $\mathbb{E}[\|x_{1,i}(t) - v_{1,i}(t)\|]\le 2H_1(t)$. Thus, combining \eqref{lower_bound} with \eqref{upper_bound}, we derive
\begin{align*}
	\bar{R}_1^{(i)}(T)&\le \frac{R_1^2}{\alpha(T)} + (L+\nu_1)^2\sum_{t=1}^T\alpha(t) + 4L\sum_{t=1}^T(H_1(t) + H_2(t)).
\end{align*}
This together with \eqref{H_l} produces \eqref{regret_bound_final}.

\noindent{\bf Proof of Theorem 3}. Taking the same idea as in the proof of Theorem \ref{thm1}, we first establish an upper bound for $\mathbb{E}\left[\sum_{t=1}^T\sum_{i=1}^{n_1}\langle g_{1,i}(t),v_{1,i}(t) - x_1\rangle\right]$. Setting $\alpha(t) = \frac{1}{\eta (t+1)}$ in \eqref{bound_2}, we obtain
\[\sum_{t=1}^T\sum_{i=1}^{n_1}\langle \hat{g}_{1,i}(t), x_{1,i}(t+1) - x_1\rangle\le \sum_{t=1}^T\sum_{i=1}^{n_1}\eta D_{\psi_1}(x_1,v_{1,i}(t)).\]
Similar to the procedure of obtaining \eqref{upper_bound_1}, we use \eqref{error_bound} to derive
\begin{equation*}
	\sum_{t=1}^T\sum_{i=1}^{n_1}\langle \hat{g}_{1,i}(t),v_{1,i}(t) - x_1\rangle\le \sum_{t=1}^T\sum_{i=1}^{n_1}\left[\eta D_{\psi_1}(x_1,v_{1,i}(t)) + \alpha(t)\|\hat{g}_{1,i}(t)\|_{\ast}^2\right].
\end{equation*}
It then follows from \eqref{unbias} that
\begin{align}
	\mathbb{E}\left[\sum_{t=1}^T\sum_{i=1}^{n_1}\langle g_{1,i}(t),v_{1,i}(t) - x_1\rangle\right]&\le \sum_{t=1}^T\sum_{i=1}^{n_1}\left(\eta \mathbb{E}[D_{\psi_1}(x_1,v_{1,i}(t))] + \alpha(t)\mathbb{E}[\|\hat{g}_{1,i}(t)\|_{\ast}^2]\right).\label{strongly_upper_bound}
\end{align}
Since $f_{1,i}$ is strongly convex with respect to $\psi_1$, by Definition \ref{strong_def}, we have
\begin{align*}
	\sum_{i=1}^{n_1}\langle g_{1,i}(t), v_{1,i}(t) - x_1\rangle - \eta D_{\psi_1}(x_1,v_{1,i}(t)) &\ge \sum_{i=1}^{n_1}[f_{1,i}(v_{1,i}(t),u_{2,i}(t)) - f_{1,i}(x_1,u_{2,i}(t))].
\end{align*}
Then, by using the analysis procedure similar to that of deriving \eqref{lower_bound}, we get
\begin{align}
	\max_{x_1\in X_1}\mathbb{E}\left[\sum_{t=1}^T\sum_{i=1}^{n_1}\langle g_{1,i}(t), v_{1,i}(t) - x_1\rangle\right] &\ge n_1\bar{R}_1^{(i)}(T) - n_1L\sum_{t=1}^T\mathbb{E}[\|x_{1,i}(t) - \bar{x}_1(t)\| + 2\|u_{2,i}(t) - \bar{x}_2(t)\|]\notag\\
	&\quad -L\sum_{t=1}^T\sum_{i=1}^{n_1}\mathbb{E}\Big[\|x_{1,i}(t) - \bar{x}_1(t)\| + \|x_{1,i}(t) - v_{1,i}(t)\|\notag\\
	&\quad + 2\|u_{2,i}(t) - \bar{x}_2(t)\|\Big] + \sum_{t=1}^T\sum_{i=1}^{n_1}\eta \mathbb{E}[D_{\psi_1}(x_1,v_{1,i}(t))].\label{strongly_lower_bound}
\end{align}
Combining \eqref{strongly_upper_bound}-\eqref{strongly_lower_bound} with Lemma \ref{lem2}, we get
\begin{align*}
	\bar{R}_1^{(i)}(T)&\le 4L\sum_{t=1}^T\sum_{l=1}^2\left(n_l(L +\nu_l)\Gamma_l\sum_{s=1}^{t-1}\theta_l^{t-1-s}\alpha(s-1) + 2(L + \nu_l)\alpha(t-1)\right) \notag\\
	&\quad + 4L\sum_{t=1}^T\sum_{l=1}^2n_l\Gamma_l\theta_l^{t-1}\Lambda_l + \sum_{t=1}^T\alpha(t)(L+\nu_1)^2.
\end{align*}
Since $\sum_{t=1}^T\frac{1}{t}\le 1 + \int_{1}^T\frac{1}{t}dt = 1 + \log T$, the above relation together with \eqref{exchange_sum} yields \eqref{regret_bound_final_1}.

\noindent{\bf Proof of Theorem 4}. According to \eqref{inner_bound_1} and the definition of $R_l^2$, we have that, for all $l\in\{1,2\}$ and $x_l\in X_l$, 
\[\sum_{s=0}^{t-1}\sum_{i=1}^{n_l}\langle \alpha(s)\hat{g}_{l,i}(s),x_{l,i}(s+1) - x_l\rangle\le n_lR_l^2.\]
Furthermore, by using a decomposition similar to \eqref{upper_bound_1}, we derive the following upper bound
\begin{equation}\label{ergo_upper_bound_1}
	\sum_{s=0}^{t-1}\sum_{i=1}^{n_l}\mathbb{E}[\langle \alpha(s)\hat{g}_{l,i}(s), v_{l,i}(s) - x_l\rangle] \le n_lR_l^2 + \sum_{s=0}^{t-1}\sum_{i=1}^{n_l}\alpha^2(s)\mathbb{E}[\|\hat{g}_{l,i}(s)\|_{\ast}^2].
\end{equation}
Construct an auxiliary sequence $\{\hat{v}_{l,i}(t)\}_{t\ge 0}$ by letting $\hat{v}_{l,i}(0) = x_{l,i}(0)$ and
\[\hat{v}_{l,i}(t) = P_{\hat{v}_{l,i}(t-1)}^l(\alpha(t)(g_{l,i}(t) - \hat{g}_{l,i}(t))),\ \forall t\ge 1\]
where $P_{\cdot}^l(\cdot)$ is the prox-mapping defined in \eqref{prox_mapping}. Then, by Assumption \ref{asm2} and Lemma 6.1 of \cite{nj:09},
\begin{align}
	\sum_{s=0}^{t-1}\sum_{i=1}^{n_l}\langle \alpha(s)(g_{l,i}(s) - \hat{g}_{l,i}(s)), \hat{v}_{l,i}(s) - x_l\rangle &\le n_lR_l^2 + \frac{n_l\nu_l^2}{2}\sum_{s=0}^{t-1}\alpha^2(s).\label{ergo_upper_bound_2}
\end{align}
Since $v_{l,i}(s)$ and $\hat{v}_{l,i}(s)$ are adapted to $\mathcal{F}_s$, it follows from an analysis similar to \eqref{unbias} that 
\begin{align}
	\mathbb{E}\left[\sum_{s=0}^{t-1}\sum_{i=1}^{n_l}\langle \alpha(s)(g_{l,i}(s) - \hat{g}_{l,i}(s)), v_{l,i}(s) - \hat{v}_{l,i}(s)\rangle\right] &= 0.\label{ergo_upper_bound_3}
\end{align}
As a combination of \eqref{ergo_upper_bound_1}-\eqref{ergo_upper_bound_3}, we obtain
\begin{equation}\label{ergo_upper_bound}
	\mathbb{E}\left[\max_{x_l\in X_l}\sum_{s=0}^{t-1}\sum_{i=1}^{n_l}\langle \alpha(s)\hat{g}_{l,i}(s), v_{l,i}(s) - x_l\rangle\right] \le 2n_lR_l^2 + n_l\left((L + \nu_l)^2 + \frac{1}{2}\nu_l^2\right)\sum_{s=0}^{t-1}\alpha^2(s).
\end{equation}
On the other hand, it follows from the convexity of $f_{1,i}(\cdot,x_2)$ that
\begin{align}
	\sum_{i=1}^{n_1}\langle \alpha(s)g_{1,i}(s), v_{1,i}(s) - x_1\rangle &\ge \sum_{i=1}^{n_1}\alpha(s)[f_{1,i}(v_{1,i}(s),u_{2,i}(s)) - f_{1,i}(x_1,u_{2,i}(s))].\label{ergo_lower_bound_1}
\end{align}
Derivation similar to \eqref{add_substract} yields
\begin{align}
	\sum_{i=1}^{n_1}f_{1,i}(v_{1,i}(s),u_{2,i}(s))&\ge n_1U(\bar{x}_1(s),\bar{x}_2(s)) - L\sum_{i=1}^{n_1}(\|u_{2,i}(s) - \bar{x}_2(s)\| + \|v_{1,i}(s) - \bar{x}_1(s)\|)
\end{align} 
Note that
\begin{align}
	-\sum_{i=1}^{n_1}f_{1,i}(x_1,u_{2,i}(s)) &= -\sum_{i=1}^{n_1}f_{1,i}(x_1,u_{2,i}(s)) + \sum_{i=1}^{n_1}f_{1,i}(x_1,\bar{x}_{2}(s))- \sum_{i=1}^{n_1}f_{1,i}(x_1,\bar{x}_{2}(s))\notag\\
	&\ge -L\sum_{i=1}^{n_1}\|u_{2,i}(s) - \bar{x}_2(s)\|- n_1U(x_1,\bar{x}_2(s)).\label{ergo_lower_bound_2}
\end{align}
By the concavity of $U(x_1,\cdot)$,
\begin{align}
	U(x_1,\hat{x}_{2,j}(t)) &\ge \frac{1}{\sum_{s=0}^{t-1}\alpha(s)}\sum_{s=0}^{t-1}\alpha(s)U(x_1,x_{2,j}(s)) \notag\\
	&= \frac{1}{\sum_{s=0}^{t-1}\alpha(s)}\sum_{s=0}^{t-1}\alpha(s)[U(x_1,x_{2,j}(s)) - U(x_1,\bar{x}_2(s)) + U(x_1,\bar{x}_2(s))]\notag\\
	&\ge \frac{1}{\sum_{s=0}^{t-1}\alpha(s)}\sum_{s=0}^{t-1}\alpha(s)[-L\|x_{2,j}(s) - \bar{x}_2(s)\| + U(x_1,\bar{x}_2(s))].\notag
\end{align}
Therefore, combining \eqref{ergo_lower_bound_1}-\eqref{ergo_lower_bound_2} and taking an ergodic average, we obtain
\begin{align}
	&\quad\frac{1}{\sum_{s=0}^{t-1}\alpha(s)}\sum_{s=0}^{t-1}\alpha(s)\frac{1}{n_1}\sum_{i=1}^{n_1}\langle \alpha(s)g_{1,i}(s), v_{1,i}(s) - x_1\rangle\notag\\
	&\ge \frac{1}{\sum_{s=0}^{t-1}\alpha(s)}\sum_{s=0}^{t-1}\alpha(s)U(\bar{x}_1(s),\bar{x}_2(s)) - U(x_1,\hat{x}_{2,j}(t))\notag\\
	&\quad - \frac{1}{\sum_{s=0}^{t-1}\alpha(s)}\sum_{s=0}^{t-1}\alpha(s)\left[\frac{1}{n_1}\sum_{i=1}^{n_1}(L(2\|u_{2,i}(s) - \bar{x}_2(s)\| + \|v_{1,i}(s) - \bar{x}_1(s)\|)) + L\|x_{2,j}(s) - \bar{x}_2(s)\|\right].\label{final_lower_bound_1}
\end{align}
In a similar way, we show that for all $x_2\in X_2$,
\begin{align}
	&\quad\frac{1}{\sum_{s=0}^{t-1}\alpha(s)}\sum_{s=0}^{t-1}\alpha(s)\frac{1}{n_2}\sum_{i=1}^{n_2}\langle \alpha(s)g_{2,i}(s), v_{2,i}(s) - x_2\rangle\notag\\
	&\ge -\frac{1}{\sum_{s=0}^{t-1}\alpha(s)}\sum_{s=0}^{t-1}\alpha(s)U(\bar{x}_1(s),\bar{x}_2(s)) + U(\hat{x}_{1,i}(t),x_2)\notag\\
	&\quad - \frac{1}{\sum_{s=0}^{t-1}\alpha(s)}\sum_{s=0}^{t-1}\alpha(s)\left[\frac{1}{n_2}\sum_{i=1}^{n_2}(L(2\|u_{1,i}(s) - \bar{x}_1(s)\| + \|v_{2,i}(s) - \bar{x}_2(s)\|)) + L\|x_{1,i}(s) - \bar{x}_1(s)\|\right].\label{final_lower_bound_2}
\end{align}
Adding \eqref{final_lower_bound_1}-\eqref{final_lower_bound_2} and utilizing \eqref{ergo_upper_bound}, we get
\begin{align*}
	&\quad\mathbb{E}\left[\max_{x_2\in X_2}U(\hat{x}_{1,i}(t),x_2) - \min_{x_1\in X_1}U(x_1,\hat{x}_{2,i}(t))\right]\\
	&\le \frac{1}{\sum_{s=0}^{t-1}\alpha(s)}\sum_{l=1}^2\left(2R_l^2 + \left((L + \nu_l)^2 + \frac{\nu_l^2}{2}\right)\sum_{s=0}^{t-1}\alpha^2(s)\right)\\
	&\quad + \frac{1}{\sum_{s=0}^{t-1}\alpha(s)}\sum_{s=0}^{t-1}4L\alpha(s)(H_1(s) + H_2(s)).
\end{align*}
Thus, the conclusion follows from Lemma \ref{lem2}.

\noindent{\bf Proof of Corollary \ref{col2}}. Let $(x_1^{\ast},x_2^{\ast})$ be the NE and note that
\begin{equation}\label{gap_lower_bound}
	\mathbb{E}\left[\max_{x_2\in X_2}U(\hat{x}_{1,i}(t),x_2) - \min_{x_1\in X_1}U(x_1,\hat{x}_{2,j}(t))\right] \ge \mathbb{E}\left[U(\hat{x}_{1,i}(t),x_2^{\ast}) - U(x_1^{\ast},\hat{x}_{2,j}(t))\right].
\end{equation}
We further have the decomposition
\begin{align*}
	\mathbb{E}\left[U(\hat{x}_{1,i}(t),x_2^{\ast}) - U(x_1^{\ast},\hat{x}_{2,j}(t))\right] &= \mathbb{E}\left[U(\hat{x}_{1,i}(t),x_2^{\ast}) - U(x_1^{\ast},x_2^{\ast}) + U(x_1^{\ast},x_2^{\ast}) - U(x_1^{\ast},\hat{x}_{2,j}(t))\right].
\end{align*}
Since $U(\cdot,\cdot)$ is $\mu$-strongly convex-strongly concave, by the definition of NE, we obtain
\[U(\hat{x}_{1,i}(t),x_2^{\ast}) - U(x_1^{\ast},x_2^{\ast})\ge \langle \partial_1U(x_1^{\ast},x_2^{\ast}),\hat{x}_{1,i}(t) - x_1^{\ast}\rangle + \frac{\mu}{2}\|\hat{x}_{1,i}(t) - x_1^{\ast}\|^2\ge \frac{\mu}{2}\|\hat{x}_{1,i}(t) - x_1^{\ast}\|^2.\]
In a similar way,
\[U(x_1^{\ast},x_2^{\ast}) - U(x_1^{\ast},\hat{x}_{2,j}(t)) \ge \frac{\mu}{2}\|\hat{x}_{2,j}(t) - x_2^{\ast}\|^2.\]
Therefore,
\[\mathbb{E}\left[U(\hat{x}_{1,i}(t),x_2^{\ast}) - U(x_1^{\ast},\hat{x}_{2,j}(t))\right]\ge \mathbb{E}\left[\frac{\mu}{2}\|\hat{x}_{1,i}(t) - x_1^{\ast}\|^2 + \frac{\mu}{2}\|\hat{x}_{2,j}(t) - x_2^{\ast}\|^2\right].\]
By Theorem \ref{thm5} and \eqref{gap_lower_bound}, we get the conclusion.

\noindent{\bf Proof of Theorem 5}. Applying Lemma \ref{lem1} to \eqref{def_x}, we get that for $l = 1,2$,
\begin{align}
	D_{\psi_l}(x_l,x_{l,i}(t+1))&\le D_{\psi_l}(x_l,v_{l,i}(t)) + \langle \alpha(t)\hat{g}_{l,i}(t),x_l - v_{l,i}(t)\rangle + \frac{1}{2}\alpha^2(t)\|\hat{g}_{l,i}(t)\|_{\ast}^2.\label{recursive}
\end{align}
By Assumption \ref{asm3} and $\sum_{i=1}^{n_l}w_{l,ij}(t) = 1$, $\sum_{i=1}^{n_l}D_{\psi_l}(x_l,v_{l,i}(t))\le \sum_{i=1}^{n_l}\sum_{j=1}^{n_l}w_{l,ij}(t)D_{\psi_l}(x_l,x_{l,j}(t)) = \sum_{i=1}^{n_l}D_{\psi_l}(x_l,x_{l,i}(t))$. It then follows from \eqref{recursive} that
\begin{align}
	\sum_{i=1}^{n_l}D_{\psi_l}(x_l,x_{l,i}(t+1))&\le \sum_{i=1}^{n_l}D_{\psi_l}(x_l,x_{l,i}(t)) + \sum_{i=1}^{n_l}\langle \alpha(t)\hat{g}_{l,i}(t),x_l - v_{l,i}(t)\rangle + \frac{1}{2}\alpha^2(t)\sum_{i=1}^{n_l}\|\hat{g}_{l,i}(t)\|_{\ast}^2\notag
\end{align}
Plugging \eqref{lower_bound_2} to this relation, we obtain
\begin{align}
	\frac{1}{n_1}\sum_{i=1}^{n_1}D_{\psi_1}(x_1, x_{1,i}(t+1)) &\le \frac{1}{n_1}\sum_{i=1}^{n_1}D_{\psi_1}(x_1, x_{1,i}(t)) + \alpha(t)(U(x_1,\bar{x}_2(t)) - U(\bar{x}_1(t),\bar{x}_2(t)))\notag\\
	&\quad + \alpha(t)L\frac{1}{n_1}\sum_{i=1}^{n_1}(\|x_{1,i}(t) - \bar{x}_1(t)\| + \|v_{1,i}(t) - x_{1,i}(t)\| + 2\|u_{2,i}(t) - \bar{x}_2(t)\|)\notag\\
	&\quad + \alpha^2(t)\frac{(L + \nu_1)^2}{2} + \alpha(t)\frac{1}{n_1}\sum_{i=1}^{n_1}\langle\hat{g}_{1,i}(t) - g_{1,i}(t),x_1 - v_{1,i}(t)\rangle. \label{network_1}
\end{align}
Similar to \eqref{lower_bound_2}, we also derive a lower bound for $\sum_{i=1}^{n_2}\langle\alpha(t)g_{2,i}(t),x_2 - v_{2,i}(t)\rangle$. Furthermore,
\begin{align}
	\frac{1}{n_2}\sum_{i=1}^{n_2}D_{\psi_2}(x_2, x_{2,i}(t+1)) &\le \frac{1}{n_2}\sum_{i=1}^{n_2}D_{\psi_2}(x_2, x_{2,i}(t)) + \alpha(t)(U(\bar{x}_1(t),\bar{x}_2(t)) - U(\bar{x}_1(t),x_2))\notag\\
	&\quad + \alpha(t)L\frac{1}{n_2}\sum_{i=1}^{n_2}(\|x_{2,i}(t) - \bar{x}_2(t)\| + \|v_{2,i}(t) - x_{2,i}(t)\| + 2\|u_{1,i}(t) - \bar{x}_1(t)\|)\notag\\
	&\quad + \alpha^2(t)\frac{(L + \nu_2)^2}{2} + \alpha(t)\frac{1}{n_2}\sum_{i=1}^{n_2}\langle\hat{g}_{2,i}(t) - g_{2,i}(t),x_2 - v_{2,i}(t)\rangle. \label{network_2}
\end{align}

Let $(x_1,x_2) = (x_1^{\ast},x_2^{\ast})$ be the NE and consider the following Lyapunov function
\begin{equation}\label{def-v}
	V(t,x_1^{\ast},x_2^{\ast}) = \frac{1}{n_1}\sum_{i=1}^{n_1}D_{\psi_1}(x_1^{\ast},x_{1,i}(t)) + \frac{1}{n_2}\sum_{i=1}^{n_2}D_{\psi_2}(x_2^{\ast},x_{2,i}(t)).
\end{equation}
Recall that $v_{l,i}(t)$, $\bar{x}_l(t)$ and $u_{l,i}(t)$ are adapted to $\mathcal{F}_t$. By adding \eqref{network_1} and \eqref{network_2}, taking conditional expectation on $\mathcal{F}_t$, and using \eqref{condition_unbias}, we obtain
\begin{align}
	\mathbb{E}[V(t+1,x_1^{\ast},x_2^{\ast})|\mathcal{F}_{t}] &\le V(t,x_1^{\ast},x_2^{\ast}) - \alpha(t)(U(\bar{x}_1(t), x_2^{\ast}) - U(x_1^{\ast},\bar{x}_2(t)))\notag\\
	&\quad + \alpha^2(t)\left(\frac{(L + \nu_1)^2}{2} + \frac{(L + \nu_2)^2}{2}\right)\notag\\
	&\quad + \alpha(t)L\sum_{l=1}^2\frac{1}{n_l}\sum_{i=1}^{n_l}e_{l,i}(t),\label{V_iter}
\end{align}
where $e_{l,i}(t) = \|x_{l,i}(t) - \bar{x}_l(t)\| + \|v_{l,i}(t) - x_{l,i}(t)\| + 2\|u_{3-l,i}(t) - \bar{x}_{3-l}(t)\|$.\\
By the definition of $H_l(t)$ in \eqref{H_l} and exchanging the order of summation, we obtain
\[\sum_{t=1}^T\alpha(t)H_l(t)\le\frac{n_l\Gamma_l\Lambda_l\alpha(0)}{1 - \theta_l} + 2(L+\nu_l)\sum_{t=1}^T\alpha^2(t-1) + \frac{n_l\Gamma_l(L + \nu_l)}{1 - \theta_l}\sum_{t=1}^T\alpha^2(t-1).\]
Therefore, it follows from $\sum_{t=1}^{\infty}\alpha^2(t) < \infty$ that $\sum_{t=1}^{\infty}\alpha(t)H_l(t) < \infty$. We further obtain $\sum_{t=1}^{\infty}\alpha(t)\mathbb{E}[e_{l,i}(t)] < \infty$ by using Lemma \ref{lem2}.
By the monotone convergence theorem, 
\[\mathbb{E}[\sum_{t=1}^{\infty}\alpha(t)e_{l,i}(t)] = \sum_{t=1}^{\infty}\alpha(t)\mathbb{E}[e_{1,i}(t)] < \infty.\]
Thus, $\sum_{t=1}^{\infty}\alpha(t)e_{l,i}(t) < \infty$ with probability $1$.
Meanwhile, note by the definition of NE that
\begin{equation}\label{U_diff}
	U(\bar{x}_1(t), x_2^{\ast})\ge U(x_1^{\ast},x_2^{\ast})\ge U(x_1^{\ast},\bar{x}_2(t)).
\end{equation}
By Lemma \ref{lem3}, $V(t,x_1^{\ast},x_2^{\ast})$ converges to a non-negative random variable with probability $1$ and 
\[0\le\sum_{t=0}^{\infty}\alpha(t)(U(\bar{x}_1(t),x_2^{\ast}) - U(x_1^{\ast},\bar{x}_2(t))) < \infty,\ a.s..\]
Also, we have
\[0\le \sum_{t=0}^{\infty}\alpha(t)\|x_{l,i}(t) - \bar{x}_l(t)\| < \infty,\ a.s.\]
Therefore, by $\sum_{t=0}^{\infty}\alpha(t) = \infty$, there exists a subsequence $\{t_r\}$ such that with probability $1$,
\[\lim_{r\to\infty}U(x_1^{\ast},\bar{x}_2(t_r)) = U(x_1^{\ast},x_2^{\ast}) = \lim_{r\to\infty}U(\bar{x}_1(t_r),x_2^{\ast}),\]
and for all $i\in\mathcal{V}_1$, $j\in\mathcal{V}_2$,
\begin{equation}\label{limit_error}
	\lim_{r\to\infty}x_{1,i}(t_r) = \lim_{r\to\infty}\bar{x}_1(t_r),\quad \lim_{r\to\infty}x_{2,j}(t_r) = \lim_{r\to\infty}\bar{x}_2(t_r).
\end{equation}
The bounded sequence $\{(\bar{x}_1(t_r),\bar{x}_2(t_r))\}$ has a convergent subsequence, and without loss of generality, we let it be indexed by the same index set $\{t_r,r = 1,2,\dots\}$. By the strict convexity-concavity of $U$, the NE is unique. Thus, according to the continuity of $U(\cdot,\cdot)$, $\bar{x}_1(t_r)\to x_1^{\ast}$ and $\bar{x}_2(t_r)\to x_2^{\ast}$ with probability $1$. Using \eqref{limit_error}, we further obtain $x_{1,i}(t_r)\to x_1^{\ast}$ and $x_{2,j}(t_r)\to x_2^{\ast}$. Therefore, by Assumption \ref{asm3} and the convergence of $V(t,x_1^{\ast},x_2^{\ast})$, $V(t,x_1^{\ast},x_2^{\ast})\to 0$ with probability $1$. Then, by \eqref{breg_prop2} and \eqref{def-v}, $x_{1,i}(t)\to x_1^{\ast}$ and $x_{2,j}(t)\to x_2^{\ast}$ with probability $1$.

\bibliographystyle{IEEEtran}
\bibliography{aistats2022.bib}
\end{document}